\documentclass[10pt, a4paper]{amsart}

\usepackage{xcolor}
\usepackage{fancyhdr}
\usepackage{tikz}
\usepackage[vmargin=100pt, hmargin=60pt]{geometry}
\usetikzlibrary{decorations.pathmorphing}
\usepackage[T1]{fontenc}
\usepackage{ucs}
\usepackage[utf8]{inputenc}
\usepackage{amsfonts}
\usepackage{listings}
\usepackage{amsmath}
\usepackage[mathscr]{euscript}
\usepackage{amssymb}
\usepackage{amsfonts}
\usepackage{dsfont}
\usepackage{amsthm}
\usepackage{fancybox}
\usepackage{enumitem}
\usepackage{dsfont}
\usepackage{algpseudocode}
\usepackage{algorithm}
\usepackage{graphicx}
\usepackage{caption}
\usepackage{subcaption}
\usepackage{float}
\usepackage{authblk}
\usepackage{relsize}
\usepackage{setspace}
\usepackage{mathtools}
\usepackage{comment}

\usepackage{fancybox} 
\usepackage{pifont} 
\usepackage[utf8]{inputenc}
\usepackage[T1]{fontenc}
\usepackage{appendix}
\usepackage{lmodern}
\usepackage{newunicodechar}
\usepackage{amsmath, amsthm, amsfonts,amssymb}
\usepackage{xcolor,mhsetup}
\usepackage[extra,safe]{tipa}
\usepackage{mathtools}
\usepackage{geometry}
\usepackage{mathrsfs}
\usepackage{hyperref}
\usepackage{url}
\usepackage{fancyhdr}
\numberwithin{equation}{section}

\def\my_c{c_\infty}

\newcommand{\mynewtheorem}[2]{
  \newaliascnt{#1}{dummy}
  \newtheorem{#1}[#1]{#2}
  \aliascntresetthe{#1}
  \expandafter\def\csname #1autorefname\endcsname{#2}
}

\newcommand{\dcb}{\begin{array}{lll}}
\newcommand{\dce}{\end{array}}
\newcommand{\ebe}{\begin{enumerate}\setlength{\baselineskip}{13pt}\setlength{\parskip}{0pt}}
\newcommand{\dbe}{\end{enumerate}}

\DeclareMathOperator*{\argsup}{argsup}

\everydisplay={\aftergroup\specindent}
\def\specindent{\global\hangindent=2em \global\hangafter=-1 \global\prevgraf=0 }

\makeatletter
\newcommand*{\inlineequation}[2][]{%
  \begingroup
    \refstepcounter{align}%
    \ifx\\#1\\%
    \else
      \label{#1}%
    \fi
    \relpenalty=100 %
    \binoppenalty= 100 %
    \ensuremath{%
      #2%
    }%
    ~\@eqnnum
  \endgroup
}
\makeatother

\newtheorem{definition}{Definition}[section]
\newtheorem{theorem}{Theorem}[section]
\newtheorem{proposition}{Proposition}[section]
\newtheorem{lemme}{Lemma}[section]
\newtheorem{corollary}{Corollary}[section]
\newtheorem{remark}{Remark}[section]
\newtheorem{example}{Example}[section]

\date{}

\usepackage{todonotes}
\begin{document}

\title{
 High order weak approximation of Stochastic Differential Equations for bounded and measurable test functions
}

\maketitle

\begin{center}
Cl\'ement Rey 
\footnote{address: CMAP, École Polytechnique, Institut Polytechnique de Paris, Route de Saclay, 91120 Palaiseau, France\\
 e-mail: clement.rey@polytechnique.edu}
\end{center}

\begin{abstract}
We present a method for approximating solutions of Stochastic Differential Equations (SDEs) with arbitrary rates. This approximation is derived for bounded and measurable test functions. Specifically, we demonstrate that, leveraging the standard weak approximation properties of numerical schemes for smooth test functions—such as first-order weak convergence for the Euler scheme—we can achieve convergence for simply bounded and measurable test functions at any desired rate by constructing a tailored approximation for the semigroup of the SDE. This is achieved by evaluating the scheme (e.g., Euler) on a random time grid. To establish convergence, we exploit the regularization properties of the scheme, which hold under a weak uniform Hörmander condition.

\end{abstract}


\tableofcontents{}

\section{Introduction}
In this article, our focus lies on the weak approximation for bounded measurable test functions of a $\mathbb{R}^{d}$-valued ($d \in \mathbb{N}^{\ast}$) random variable $X_{T}$, $T \geqslant 0$ where $(X_{t})_{t \geqslant 0}$ is a solution to the inhomogeneous Stochastic Differential Equation (SDE)
\begin{align}
\label{eq:SDE_inv}
X_{t}=X_{0}+\int_{0}^{t} V_{0}(X_{s},s) \mbox{d}s+ \sum_{i=1}^{N} \int_{0}^{t} V_{i}(X_{s},s) \mbox{d}W^{i}_{s}, \quad t \geqslant 0,X_{0}\in \mathbb{R}^{d},
\end{align}
where $((W^{i}_{t})_{t \geqslant 0}, i \in \{1,\ldots,N\})$ are $N \in \mathbb{N}^{\ast}$ independent $\mathbb{R}$-valued standard Brownian motions and $V_{j} \in \mathcal{C}^{\infty}_{b}(\mathbb{R}^{d};\mathbb{R}^{d})$, $j \in \{0,\ldots,d\}$.  When for every $j \in \{1,\ldots,d\}$, $V_{j}$ does not depend on the time, $i.e.$ its second variable, the SDE is termed homogeneous.\\

More precisely, for any chosen $\nu>0$, $T>0$ and sufficiently large $n \in \mathbb{N}^{\ast}$, we demonstrate that for an approximation functional operator $\hat{Q}^{\nu,\frac{T}{n}}_{0,T}$ which will be made explicit later, there exists $C>0$ such that for every measurable and bounded function $f$,
\begin{align}
\label{eq:intro_TV}
\sup_{x \in \mathbb{R}^{d}} \vert \mathbb{E}[f(X_{T}) \vert X_{0}=x]- \hat{Q}^{\nu,\frac{T}{n}}_{0,T}f(x)\vert \leqslant  \frac{C}{n^{\nu}} \Vert f \Vert_{\infty},
\end{align}
where $\Vert f \Vert_{\infty}=\sup_{x \in \mathbb{R}^{d}} \vert f(x) \vert$.  We recall that for $\mu_{1}$ and $\mu_{2}$ two probability measures on $\mathbb{R}^{d}$, the total variation distance between $\mu_{1}$ and $\mu_{2}$ is given by
\begin{align*}
d_{TV}(\mu_{1} , \mu_{2} ) =  \sup_{A \in \mathcal{B}(\mathbb{R}^{d})} \vert \mu_{1}(A) - \mu_{2}(A) \vert  = & \sup_{f \in \mathcal{M}(\mathbb{R}^{d};\mathbb{R}), \Vert f \Vert_{\infty} \leqslant 1} \frac{1}{2} \vert \mu_{1}(f) - \mu_{2}(f) \vert  \\
= & \sup_{f \in \mathcal{C}^{\infty}_{c}(\mathbb{R}^{d};\mathbb{R}), \Vert f \Vert_{\infty} \leqslant 1} \frac{1}{2} \vert \mu_{1}(f) - \mu_{2}(f) \vert  
\end{align*}
where $\mu_{1}(f) = \int_{\mathbb{R}^{d}} f(y) \mu_{1}( \mbox{d} y)$ and similarly for $\mu_{2}(f)$. The last equality above is a direct consequence of the Lusin's Theorem. In our case, given $x \in \mathbb{R}^{d}$, $\hat{Q}^{\nu,\frac{T}{n}}_{0,T}(x, \mbox{d} y)$ is not necessarily a probability measure (it is finite with $\hat{Q}^{\nu,\frac{T}{n}}_{0,T}(x, \mathbb{R}^{d})=1$ but not necessarily positive). Nevertheless, we will refer to (\ref{eq:intro_TV}), by abuse of language,  as, for every $x$, the total variation convergence of $\hat{Q}^{\nu,\frac{T}{n}}_{0,T}(x, \mbox{d} y)$ towards $P_{T}(x,\mbox{d}y)$ that is the probability measure of $X_{T}$ starting from $x$ at time 0.\\

The algorithm utilized to construct $\hat{Q}^{\nu,\frac{T}{n}}_{0,T}$ is adapted from the one that was presented in \cite{Alfonsi_Bally_2019} and is based on a combination of discrete time approximations $X^{\Pi_{h}}_{T}$ of (\ref{eq:SDE_inv}) which are built upon random time grids $\Pi_{h}$, $h \in \{1,\ldots,r\}$, $r\in \mathbb{N}^{\ast}$. More specifically, given $h$, we consider a sequence of
independent random variables $U^{h}_{j}\in \mathbb{R}^{N},\; j\in
\mathbb{N}^{\ast}$, and we assume that $U^{h}_{j}$, are centered with covariance one. We construct the $\mathbb{R}^{d}$-valued process $(X^{\Pi_{h}}_{t})_{t \in \pi^{\delta}}$ in the following way:%
\begin{align}
X^{\Pi_{h}}_{t_{j+1}^{\Pi_{h}}}=\psi(X^{\Pi_{h}}_{t_{j}^{\Pi_{h}}},t_{j}^{\Pi_{h}},(t_{j+1}^{\Pi_{h}}-t_{j}^{\Pi_{h}})^{\frac{1}{2}} U^{h}_{j+1}, t_{j+1}^{\Pi_{h}}-t_{j}^{\Pi_{h}}) , \quad t \in  \pi^{\delta}, \quad X^{\delta}_{0}=\mbox{\textsc{x}}^{\delta}_{0}\in \mathbb{R}^d \label{eq:schema_general_intro}
\end{align}%
for $\Pi_{h}=\{t_{j}^{\Pi_{h}},j \in \mathbb{N}\}$ with
\begin{align*}
\psi \in \mathcal{C}^{\infty }( \mathbb{R}^{d}\times \mathbb{R}_+\times \mathbb{R}^{N} \times [0,1];\mathbb{R}^{d})\quad \mbox{and} \quad  \forall (x,t) \in \mathbb{R}^{d} \times \pi^{\delta},\psi
(x,t,0,0)=x.  
\end{align*}
Considering for instance the Euler scheme of (\ref{eq:SDE_inv}), we have
\begin{align*}
\psi(x,t,z,y)=x+V_{0}(x,t)y+\sum_{i=1}^{N}V_{i}(x,t)z
\end{align*}

Moreover, $\Pi_{l}$ is independent from $((U^{h}_{j})_{j \in \mathbb{N}^{\ast}})_{h \in \{1,\ldots,r\}}$. The time $t_{j+1}^{\Pi_{h}}-t_{j}^{\Pi_{h}}$ between two successive discrete time values is chosen randomly in $\{ \frac{T}{n^{l}}, l\in \mathbb{N}\}$.  It is shown in \cite{Alfonsi_Bally_2019} that we can use the following representation: For every $x \in \mathbb{R}^{d}$,
\begin{align}
\label{eq:rep_formula_semigroup}
\hat{Q}^{\nu,\frac{T}{n}}_{0,T}f(x)= \sum_{h=1}^{r}c_{h} \mathbb{E}[f(X^{\Pi_{h}}_{T}) \vert X^{\Pi_{h}}_{0}=x]
\end{align}
where the value of $c_{h}\in \mathbb{R}$ and the law of $\Pi_{h}$ are given explicitly. This writting allows to compute $\hat{Q}^{\nu,\frac{T}{n}}_{0,T}f(x)$ by a Monte Carlo approach sampling $M \in \mathbb{N}^{\ast}$ independent realizations of $X^{\Pi_{h}}_{T}$, $h \in \{1,\ldots,r\}$. 

 In addition to the this representation (\ref{eq:rep_formula_semigroup}), the authors of \cite{Alfonsi_Bally_2019} proved that (\ref{eq:intro_TV}) hold but with $\Vert f \Vert_{\infty}$ replaced by $\Vert f \Vert_{\infty,K}$ the supremum norm of $f$ but also of its derivatives up to order $K \in \mathbb{N}^{\ast}$, $i.e.$
 \begin{align}
 \label{eq:intro_smooth}
\sup_{x \in \mathbb{R}^{d}} \vert \mathbb{E}[f(X_{T}) \vert X_{0}=x]- \hat{Q}^{\nu,\frac{T}{n}}_{0,T}f(x)\vert \leqslant  \frac{C}{n^{\nu}} \Vert f \Vert_{\infty,K}.
 \end{align}
Furthermore, the computational complexity for computing $\hat{Q}^{\nu,\frac{T}{n}}_{0,T}$ using $M$ Monte Carlo samples is of order $C_{rand}=M \times n \times r \times C(\nu)$ where $C(\nu)$ depends on $\nu$. Notice that a standard approach on deterministic time grid of size $\frac{T}{n}$ is $C_{det}=M \times n$ and then $C_{det}= \underset{n \to \infty}{\mathcal{O}}(C_{rand})$ underscoring the great numerical interest of this random grid approach as $\nu$ can be arbitrarily high.  \\
  
  The main interest of this method arise for big $\nu$. Actually, for small $\nu$, (\ref{eq:intro_smooth}) is usually obtained by a Lindeberg type approach relying on a smooth short time approximation of the form
  \begin{align}
  \label{eq:short_time_error_intro}
 \vert \mathbb{E}[f(X_{t+\delta})-f(\psi(x,t,\sqrt{\delta} U,\delta )) \vert X_{t}=x ]\vert  \leqslant C \Vert f\Vert _{\infty,\beta} \delta^{\alpha+1} 
\end{align}
 where $\alpha>0$ is referred in this paper as the weak smooth order of the scheme $\psi$. If $\alpha \geqslant \nu$, one can simply chose $\hat{Q}^{\nu,\frac{T}{n}}_{0,T}=Q^{\frac{T}{n}}_{0,T}$ where for every measurable $f$ and every $x \in \mathbb{R}^{d}$, $Q^{\frac{T}{n}}_{s,t}f(x)=\mathbb{E}[f(X^{\frac{T}{n}}_{t})\vert X^{\frac{T}{n}}_{s}=x ] $, $s,t \in \pi^{\frac{T}{n}}=\{ k \frac{T}{n}, k \in \mathbb{N} \}$ is the semigroup arising from the Markov process $X^{\frac{T}{n}}$ defined on the homogeneous deterministic time grid with time step $\frac{T}{n}$ (we use $X^{\frac{T}{n}}$ as short notation for $X^{\pi^{\frac{T}{n}}}$ defined as in (\ref{eq:schema_general_intro}) with $\Pi_{h}=\pi^{\frac{T}{n}}$). In the case of the Euler scheme we have $\alpha=1$ (see \cite{Talay_Tubaro_1991}), but various higher weak smooth order methods exists (see $e.g.$ \cite{Talay_1990}, \cite{Ninomiya_Victoir_2008}, \cite{Alfonsi_2010},\cite{Rey_2017}). However the value of $\alpha$ remains limited or requires high computational complexity to be increased. \\
 
 The algorithm in \cite{Alfonsi_Bally_2019} exploits (\ref{eq:short_time_error_intro}) on random time intervals and combine schemes on random time grids to build $\hat{Q}^{\nu,\frac{T}{n}}$ depending on $\alpha$ and $\nu$ and ensuring that (\ref{eq:intro_smooth}) holds. The combination of approximation method to  enhance convergence rates also appears in a large scope of method for the computation of $ \mathbb{E}[f(X_{T}) \vert X_{0}=x]$.  The Multi-Level Monte Carlo (see \cite{Heinrich_2001} and \cite{Giles_2008}), which extends the statistical Romberg \cite{Kebaier_2005}, exploits combination of scheme on different and specific time step to achieve high order approximation. In a similar way, the Richardson Romberg build high order approximation by taking advantage of development of the weak error (see $e.g.$ \cite{Pages_2007}). Actually, Multi Level Monte Carlo and Richardson Romberg methods can be combine for even greater efficiency as demonstrated in \cite{Lemaire_Pages_2017}.For more details about the methods related to those approaches, we refer to the non-exhaustive list of works \cite{McLeish_2011}, \cite{BenAlaya_Kebaier_2015}, \cite{Glynn_Rhee_2015},  \cite{AlGerbi_Jourdain_Clement_2016}, \cite{Vihola_2018}.\\

 Therefore, the aim of this article is to show that the result from \cite{Alfonsi_Bally_2019} (see (\ref{eq:intro_smooth})) is true with $K=0$.  When dealing with total variation distance for schemes on deterministic homogeneous time grids, some results already exists. Essentially, those results prove the convergence in total variation with rate $\frac{1}{n^{\alpha}}$ considering that (\ref{eq:short_time_error_intro}) is satisfied. In comparison to the case $K \in \mathbb{N}^{\ast}$, the proof of the total variation convergence requires regularization properties on the semigroups arising from $X$ or $X^{\frac{T}{n}}$, which are themselves obtained using ellitpic or H\"ormander type assumptions.
  For instance, \cite{Bally_Talay_1996_I} addresses the scenario where $\psi$ is the Euler scheme of an homogeneous SDE satisfying weak uniform H\"ormander property. They also propose an expansion of the error for bounded and measurable test functions which enable to use some Richardson Romberg methods. Regarding generic schemes (which encompass the Euler scheme but also many more buildings) of  inhomogeneous SDE which are simply specified by a transition function - such as $X^{\frac{T}{n}}$ -,  some results were also already established. In \cite{Bally_Rey_2015}, the uniform elliptic case is studied, while in \cite{Rey_2024}, the weak local H\"ormander case is addressed.  In this case, an additional hypothesis with form (\ref{eq:short_time_error_intro}) is needed to reach order $\alpha$ for the total variation convergence. It is worth noticing that it is shown in \cite{Rey_2024} that, even without assuming (\ref{eq:short_time_error_intro}), the convergence in total variation error of such generic schemes to $X_{T}$ happens with at least order $\frac{1}{n^{\frac{1}{2}-\epsilon}}$, for any $\epsilon>0$, under the same assumptions that guarantee regularization properties. \\
  
  However, the total variation order of convergence cannot go further standard weak smooth order $\alpha$ of the scheme given by  (\ref{eq:short_time_error_intro}).  As an illustration, the Euler scheme cannot converge faster that $\frac{1}{n}$ for the total variation distance and the only way to improve the convergence remains to increase the value of $n$. In this article we show how, without increasing $n$, the method outlined in \cite{Alfonsi_Bally_2019} can be applied to improve the total variation order of convergence to order $\frac{1}{n^{\nu}}$ for any $\nu$ even when considering only scheme of weak smooth order $\alpha$ when $\alpha<\nu$ ($e.g.$ $\alpha=1$ for the Euler scheme).\\
 
 In this paper,we do not discuss the representation formula (\ref{eq:rep_formula_semigroup}). We refer to an alternative equivalent representation, which only involves the discrete time semigroups
\begin{align*}
Q^{\frac{T}{n^{l}}}_{s,t}f(x)=\mathbb{E}[f(X^{\frac{T}{n^{l}}}_{t}) \vert X^{\frac{T}{n^{l}}}_{s} = x], \quad \forall s,t \in \pi^{\frac{T}{n^{l}}},
\end{align*}
where $X^{\frac{T}{n^{l}}}$ is a short notation for $X^{\pi^{\frac{T}{n^{l}}}}_{t}$ ($\pi^{\frac{T}{n^{l}}}=\{ k \frac{T}{n^{l}}, k \in \mathbb{N} \}$).  The aformentionned representation formula is given in (\ref{eq:schema_ordre_quelconque}).  First, we show that (\ref{eq:intro_TV}) can be obtained using an abstract Lindeberg inspired decomposition and assuming regularization properties on $Q^{\frac{T}{n^{l}}}_{s,t}$ (see Theorem \ref{th:intro_erreur_faible} and Theorem \ref{theo:distance_density}). Then, we propose sufficient assumption on the function $\psi$ and $((U^{h}_{j})_{j \in \mathbb{N}^{\ast}})_{h \in \{1,\ldots,r\}}$ such that the required regularization properties of $Q^{\frac{T}{n^{l}}}_{s,t}$ hold under a weak uniform type H\"ormander assumption. Those properties were demonstrated in \cite{Rey_2024} and are restated in our context in Proposition \ref{th:regul_main_result_intro}.\\

The article is organized as follows. Section 2 presents the abstract Lindeberg framework to derive (\ref{th:regul_main_result_intro}) from regularization properties on $Q^{\frac{T}{n^{l}}}_{s,t}$. The main results of this section are gathered in Theorem \ref{th:intro_erreur_faible} and Theorem \ref{theo:distance_density}. In Section 3, we state our main result of the article concerning total variation convergence under suitable hypothesis on $\psi$ and $((U^{h}_{j})_{j \in \mathbb{N}^{\ast}})_{h \in \{1,\ldots,r\}}$ which is given in Theorem \ref{th:main_result_psi_CVTV}.

\section{The distance between Semigroups}
\label{Section:The distance between two Markov semigroups}
Throughout this paper the following notations will prevail. We fix $T>0$ and $n \in \mathbb{N}^*$. For $l \in \mathbb{N}$ we will denote $\delta_{n}^{l}=T/n^{l}$ and for $\delta>0$ we consider the time grid $\pi^{\delta}:= \{k \delta, k \in \mathbb{N}\}$, with the convention $\pi^0=\mathbb{R}_+$. 

\subsection{Framework}\;\\
\textbf{Notations.} 
For $d \in \mathbb{N}^{\ast}$, denote by
\begin{itemize}
\item $\mathcal{M}_b(\mathbb{R}^d)$, the set of measurable and bounded functions from $\mathbb{R}^d$ to $\mathbb{R}$.
\item $\mathcal{C}^q(\mathbb{R}^d) $, $q \in \mathbb{N} \cup \{+\infty\}$, the set of functions from $\mathbb{R}^d$ to $\mathbb{R}$ which admit derivatives up to order $q$ and such that all those derivatives are continuous.
\item $\mathcal{C}_b^q(\mathbb{R}^d) $, $q \in \mathbb{N} \cup \{+\infty\}$, the set of functions from $\mathbb{R}^d$ to $\mathbb{R}$ which admit derivatives up to order $q$ and such that all those derivatives are bounded.
\item $\mathcal{C}_1^q(\mathbb{R}^d) $, $q \in \mathbb{N} \cup \{+\infty\}$, the set of functions from $\mathbb{R}^d$ to $\mathbb{R}$ which admit derivatives up to order $q$ and such that all those derivatives are bounded in $\mbox{L}_1\left(\mathbb{R}^d\right)$.
\item $\mathcal{C}_c^q(\mathbb{R}^d) $, $q \in \mathbb{N} \cup \{+\infty\}$, the set of functions from $\mathbb{R}^d$ to $\mathbb{R}$ defined on compact support and which admit derivatives up to order $q$.
\end{itemize}
For a multi-index $\gamma =(\gamma^{1},\cdots,\gamma^{d})\in \mathbb{N}^{d}$ we denote $\vert \gamma \vert
=\gamma^{1}+...+\gamma^{d}$ and if $f\in \mathcal{C}^{\vert \gamma \vert}(\mathbb{R}^{d})$,  we define $\partial^{\gamma}_{x}f=(\partial_1)^{\gamma^{1}} \ldots (\partial_{d})^{\gamma^{d}} f =\partial
_{x^{1}}^{\gamma^{1}}\ldots\partial_{x^{d}}^{\gamma^{d}}f(x).$\\

\textbf{Framework.} The approach we propose consists in building an approximation for a family of semigroups $\left(P^{\delta}_{t,s})_{t ,s \in \pi^{\delta} ;t \leqslant s} \right)_{\delta >0}$ where we suppose that this family is independent to the time-grid $\pi^{\delta}$ in the following sense

\begin{center}
For every $\delta,\overline{\delta}>0, t,s \in \pi^{\delta} \cap \pi^{\overline{\delta}}$, we have $P^{\delta}_{s,t}=P^{\overline{\delta}}_{s,t}=:P_{s,t}$. 
\end{center}
The crucial property satisfiedby $P$ is called the semigroup property and write : For every $s \leqslant u \leqslant t$, $P_{s,u} P_{u,t}= P_{s,t}$. At this point, we have in mind that $P$ may be the semigroup of an inhomogeneous Markov process $(X_{t})_{t \geqslant 0}$ solution to (\ref{eq:SDE_inv}), such that for every measurable function $f : \mathbb{R}^{d} \to \mathbb{R}$ and every $x \in \mathbb{R}^{d}$, $P_{s,t}f(x)=\mathbb{E}[f(X_{t}) \vert X_{s}=x]$. Now, given a value for $\delta$ we intrduce an approximation process $(X^{\delta})_{t \in \pi^{\delta}}$ for $(X_{t})_{t \geqslant 0}$ which is supposed to statfisfy the Markov property. In particular, denoting $Q^{\delta}_{s,t}=\mathbb{E}[f(X^{\delta}_{t}) \vert X^{\delta}_{s}=x]$, we have $Q^{\delta}_{s,u} Q^{\delta}_{u,t}= Q^{\delta}_{s,t}$ for every $s \leqslant u \leqslant t$.  Notice that $Q^{\delta}$ may be used directly to approximate $P$ but we are going a step further. Actually, the approximation we consider for $P$ will be a concatenation of some $Q^{\delta}$ but involving different possible values of $\delta \in \{\delta^{l}_{n}, l \in \mathbb{N}^{\ast} \}$.

\subsection{Approximation results}
\subsubsection{Arbitrary order weak approximation.}
We introduce the following assumptions
In a first step, we suppose that the semigroup we study is such that for every $r\in \mathbb{N}$, if $f\in
\mathcal{C}_b^{r}(\mathbb{R}^{d})$ then $P_{s,t}f\in \mathcal{C}_b^{r}(\mathbb{R}^{d})$ and%
\begin{align}
\sup_{t \geqslant s \geqslant 0}\Vert P_{s,t}f\Vert _{r,\infty }\leqslant C\Vert f\Vert _{r,\infty }.  
\label{hyp:transport_regularite_semigroup} 
\end{align}

The approximation of $P$ we consider in this paper is built from a family of discrete semigroups $Q:=\left( Q^{\delta_{n}^{l}} \right)_{l \in \mathbb{N}}=\left( ( Q^{\delta_{n}^{l}}_{s,t} )_{s,t \in \pi^{\delta_{n}^{l}};s \leqslant t} \right)_{l \in \mathbb{N}}$ such that for every $f\in
\mathcal{C}_b^{r}(\mathbb{R}^{d})$ then $Q^{\delta_{n}^{l}}_{s,t}f\in \mathcal{C}_b^{r}(\mathbb{R}^{d})$ and
\begin{align}
 \forall s,t \in \pi^{\delta_{n}^{l}}, s \leqslant t ,\quad  \Vert Q^{\delta_{n}^{l}}_{s,t}f\Vert
_{r,\infty}\leqslant C\Vert f\Vert _{r,\infty}. 
\label{hyp:transport_regularite_semigroup_approx} 
\end{align}

and which satisfy the short-time estimate, for every $r,\beta \in \mathbb{N}$, $\alpha>0$, and $f\in
\mathcal{C}_b^{\beta+r}(\mathbb{R}^{d})$
\begin{align}
E_n(l,\alpha,\beta,P,Q)  \qquad  \forall t \in \pi^{\delta_{n}^{l}} , \quad   \Vert P_{t,t+\delta_{n}^{l}}f-Q^{\delta_{n}^{l}}_{t,t+\delta_{n}^{l}}f \Vert_{\infty,r}   \leqslant C \Vert f\Vert _{\infty,\beta+r} \left( \delta_{n}^{l} \right)^{\alpha+1}.  \label{hyp:erreur_tems_cours_fonction_test_reg}
\end{align}

Using the family $\left( ( Q^{\delta_{n}^{l}}_{s,t} )_{s,t \in \pi^{\delta_{n}^{l}};s \leqslant t} \right)_{l \in \mathbb{N}}$, for every $(l,\nu) \in \mathbb{N}^2$, we are going to build $(\hat{Q}^{\nu,\delta_{n}^{l}}_{t,t+\delta_{n}^{l}})_{t \in \pi^{\delta_{n}^{l}}}$ as an approximation of $(P_{t,t+\delta_{n}^{l}})_{t \in \pi^{\delta_{n}^{l}}}$ which, under the hypothesis (\ref{hyp:transport_regularite_semigroup_approx}) and (\ref{hyp:erreur_tems_cours_fonction_test_reg}), satisfies, for every $r \in \mathbb{N}$, and $f\in
\mathcal{C}_b^{\kappa(l,\nu)+r}(\mathbb{R}^{d})$
\begin{align}
\hat{E}_n(l,\nu,\kappa,P,Q)  \qquad  \forall t \in \pi^{\delta_{n}^{l}} , \quad   \Vert P_{t,t+\delta_{n}^{l}}f-\hat{Q}^{\nu,\delta_{n}^{l}}_{t,t+\delta_{n}^{l}}f \Vert_{\infty,r}   \leqslant C \Vert f\Vert _{\infty,\kappa(l,\nu)+r} \frac{1+T^{\frac{\nu+1}{l+1}}}{n^{\nu}},
\label{hyp:erreur_tems_cours_fonction_test_regschema_nu} 
\end{align}
with 
\begin{align*}
\kappa(l,\nu)= \max \left\{ \beta m(l,v), \max_{i=1}^{m(l,v)-1} \left\{ i \kappa(l+1,q_{i}(l,v)) \right\} \right\}
\end{align*}
and
\begin{align*}
m\left( l,\nu \right) =& \left \lceil  \frac{\nu}{(1+\alpha)l + \alpha } \right \rceil \\
q_i\left( l,\nu \right) =& \nu + \lceil  i - (1+\alpha)(l+1)(i-1)  \rceil, \quad \forall i \in \{1,\ldots,m\left( l,\nu \right) - 1\}  .
\end{align*}
and $p(l,\nu) \geqslant 0$ is a positive constant depending on $l$, $\nu$ and $\alpha$. \\

 In particular $\hat{Q}^{\nu,\delta_{n}^{0}}_{0,\delta_{n}^{0}}$ is an approximation of $P_{T}$ with accuracy $1/n^{\nu}$. The approach we use was first introduced in \cite{Alfonsi_Bally_2019}. Among other, it was shown in this paper that, combined with a random grid approach, the accuracy $1/n^{\nu}$ can be reached with complexity - in terms of the number of simulations of random variables with law given by a semigroup $Q^{\delta_{n}^{l}}$ for some $l \in \mathbb{N}$ - of order $n$. We do not discuss the simulation of $\hat{Q}^{\nu,\delta_{n}^{0}}_{0,\delta_{n}^{0}}$ in this paper and we invite the reader to refer to \cite{Alfonsi_Bally_2019} for more details.
 

For every $t \in \pi^{\delta_{n}^{l}}$, we define $\hat{Q}^{\nu,\delta_{n}^{l}}$ in the following recursive way 

\begin{align}
\label{eq:schema_ordre_quelconque}
\hat{Q}^{\nu,\delta_{n}^{l}}_{t,t+\delta_{n}^{l}}=& Q^{\delta_{n}^{l+1}}_{t,t+\delta_{n}^{l}}+\sum_{i=1}^{m\left( l,\nu \right)-1} \hat{I}^{\delta_{n}^{l+1}}_{t,t+\delta_{n}^{l},i} ,
\end{align}
with 
\begin{align*}
\hat{I}^{\delta_{n}^{l+1}}_{t,t+\delta_{n}^{l},i} =&   \sum_{t=t_0 < t_1<\ldots<t_{i}\leqslant t+\delta_{n}^{l} \in \pi^{\delta_{n}^{l+1}}}  \prod_{j=1}^{i} Q^{\delta_{n}^{l+1}}_{t_{j-1},t_{j}-\delta_{n}^{l+1}}   \left( \left( \hat{Q}^{q_i\left( l,\nu \right),\delta_{n}^{l+1}}_{t_{j}-\delta_{n}^{l+1},t_{j}} - Q^{\delta_{n}^{l+1}}_{t_{j}-\delta_{n}^{l+1},t_{j}}   \right) \right) Q^{\delta_{n}^{l+1}}_{t_{i},t+\delta_{n}^{l}}  .
\end{align*}
Notice that the recursion ends when $m\left( l,\nu \right)=1$ and in this case $\hat{Q}^{\nu,\delta_{n}^{l}}_{t,t+\delta_{n}^{l}}= Q^{\delta_{n}^{l+1}}_{t,t+\delta_{n}^{l}}$. When $m\left( l,\nu \right)>1$, the recursion still ends for every $(l,\nu) \in \mathbb{N}^2$ and $\hat{Q}^{\nu,\delta_{n}^{l}}_{t,t+\delta_{n}^{l}}$ is well-defined. This is a direct consequence of Lemme 3.8 in \cite{Alfonsi_Bally_2019}.  We also invite the reader to refer to this article for the proof of $\hat{E}_n(l,\nu,\kappa,P,Q)$ for every $(l,\nu) \in \mathbb{N}^2$. In particular, $\hat{Q}^{\nu,\delta_{n}^{0}}_{0,\delta_{n}^{0}}=\hat{Q}^{\nu,\delta_{n}^{0}}_{0,T}$ is well defined, satisfies $\hat{E}(0,\nu,\kappa,P,Q)$ and may be built from the family $\left(( Q^{\delta_{n}^{l}}_{t})_{t \in \pi^{\delta_{n}^{l}}} \right)_{l \in \{1,\ldots,l(\nu,\alpha) \}}$ with $l(\nu,\alpha)=\lceil  \nu / \alpha  \rceil$

\subsubsection{Arbitrary order total variation converge.}Our purpose is to obtain a similar estimation as $\hat{E}(0,\nu,\kappa,P,Q)$ for $\Vert P_{T}f-\hat{Q}^{\nu,\delta_{n}^{0}}_{0,T}f \Vert_{\infty}$ but which remains valid for simply bounded and measurable test functions $f$. In other words we want to show that $\hat{E}(0,\nu,0,P,Q)$ holds. We will obtain such results using a dual approach. In particular, for a functional operator $Q$, we denote by $Q^{\ast}$ its dual operator for the scalar product in $\mbox{L}^{2}(\mathbb{R}^d)$ (\textit{i.e.} $%
\left\langle Qg,f\right\rangle_{\mbox{L}^{2}(\mathbb{R}^d)} =\left\langle
g,Q^{\ast}f\right\rangle_{\mbox{L}^{2}(\mathbb{R}^d)} $). Our approach requires to introduce some additional assumptions concerning our discrete semigroups. A first step is to consider a dual version of (\ref{hyp:transport_regularite_semigroup}) and (\ref{hyp:transport_regularite_semigroup_approx}). We assume that for every $f\in
\mathcal{C}_1^{r}(\mathbb{R}^{d})$, then $P^{\ast}_{s,t}f \in \mathcal{C}_1^{r}(\mathbb{R}^{d})$ and
\begin{align}
\sup_{t \geqslant s \geqslant 0}\Vert P^{\ast}_{s,t}f\Vert _{r,1 }\leqslant C\Vert f\Vert _{r,1}.  
\label{hyp:transport_regularite_semigroup_dual} 
\end{align}
and for the family of semigroups $Q=\left( \left( Q^{\delta_{n}^{l}}_{s,t} \right)_{s,t \in \pi^{\delta_{n}^{l}};s \leqslant t} \right)_{l \in \mathbb{N}}$,  $Q^{\delta_{n}^{l},\ast}_{s,t}f \in \mathcal{C}_1^{r}(\mathbb{R}^{d})$ and
\begin{align}
 \forall s,t \in \pi^{\delta_{n}^{l}}, s \leqslant t ,\quad  \Vert Q^{\delta_{n}^{l},\ast}_{s,t}f\Vert
_{r,1}\leqslant C\Vert f\Vert _{r,1}. 
\label{hyp:transport_regularite_semigroup_approx_dual} 
\end{align}
Moreover, we assume that the following dual estimate of the error in short time holds: for every $r \in \mathbb{N}$, and $f\in
\mathcal{C}_1^{\beta+r}(\mathbb{R}^{d})$
%

\begin{align}
E_n(l,\alpha,\beta,P,Q)^{\ast}  \qquad  \forall t \in \pi^{\delta_{n}^{l}} , \quad   \Vert P^{\ast}_{t,t+\delta_{n}^{l}}f-Q^{\delta_{n}^{l},\ast}_{t,t+\delta_{n}^{l}}f    \Vert_{r,1} \leqslant C\Vert f\Vert_{\beta+r,1}\left( \delta_{n}^{l} \right)^{\alpha+1}.  \label{hyp:erreur_tems_cours_fonction_test_reg_dual}
\end{align}

At this point, we notice that using the same approach as in \cite{Alfonsi_Bally_2019}, we can derive from (\ref{hyp:transport_regularite_semigroup_dual}), (\ref{hyp:transport_regularite_semigroup_approx_dual}) and (\ref{hyp:erreur_tems_cours_fonction_test_reg_dual}),  that for every $(l,\nu) \in \mathbb{N}^{\ast}$, and $f\in
\mathcal{C}^{\kappa(l,\nu)+r}(\mathbb{R}^{d})$
\begin{align}
\hat{E}_n(l,\nu,\kappa,P,Q)^{\ast}  \qquad  \forall t \in \pi^{\delta_{n}^{l}} , \quad   \Vert P^{\ast}_{t,t+\delta_{n}^{l}}f-\hat{Q}^{\nu,\delta_{n}^{l},\ast}_{t,t+\delta_{n}^{l}}f    \Vert_{r,1} \leqslant C \Vert f\Vert_{\kappa(l,\nu)+r,1} \frac{1+T^{p(l,\nu)}}{n^{\nu}}.  \label{hyp:erreur_tems_cours_fonction_test_reg_schema_nu}
\end{align}
where $p(l,\nu) \geqslant 0$ is a positive constant depending on $l$, $\nu$ and $\alpha$.\\

Now, we introduce some regularization properties that will be necessary to obtain total variation convergence. In concrete applications, the property is not necessarily satisfied by the discrete semigroup $Q^{\delta}$ but by a family of functional operators close enough in total variation distance. We call this family, a modification of $Q^{\delta}$ and it is not necessarily a semigroup. Hence, this hypothesis is expressed not only for discrete semigroups but for discrete family of functional operators. \\

Let $ \; q \in \mathbb{N}$ and $\eta \geqslant 0$ be fixed. For $\delta>0$, let $\left(Q ^{\delta}_{s,t}\right)_{s,t \in \pi^{\delta} ,t>s}$, be a family of functional operators. We consider the following regularization property : \\

For every $r \in \mathbb{N}$ and every multi-index $\gamma $ with $\vert \gamma \vert
+r \leqslant q$, and $f \in \mathcal{C}_{b}^{\infty}(\mathbb{R}^{d})$, then
\begin{align}
R_{q,\eta }(Q^{\delta})  \qquad  \forall s,t \in \pi^{\delta} ,t>s, \quad  \Vert Q^{\delta}_{s,t} \partial_{x}^{\gamma}f\Vert
_{r,\infty}\leqslant \frac{C}{(t-s)^{\eta}}\Vert f\Vert _{\infty}.  
\label{hyp:reg_forte}
\end{align}

In our approach we will not necessarily use directly (\ref{hyp:reg_forte}) but an estimate it implies on the adjoint semigroup. Actually, remembering that $\Vert f \Vert_1 \leqslant \sup_{g \in \mathcal{M}_{b}, \Vert g \Vert_{\infty}=1} \langle g,f \rangle$, we notice that $R_{q,\eta }(Q^{\delta})$ implies that for every $r \in \mathbb{N}$ and every multi-index $\gamma $ with $\vert \gamma \vert
+r \leqslant q$ and $f\in
\mathcal{C}^{\infty}(\mathbb{R}^{d}) \cap \mbox{L}^{1}(\mathbb{R}^{d})$
\begin{align}
\label{hyp:reg_forte_dual}
\forall t \in \pi^{\delta_{n}^{l}} ,t>s, \quad  \Vert Q^{\delta,\ast}_{s,t} \partial_{x}^{\gamma}f\Vert
_{r,1}\leqslant \frac{C}{(t-s)^{\eta}}\Vert f\Vert _{1}. 
\end{align}

Using those hypothesis, we can derive the following total variation convergence towards the semigroup $P$ with rate $1/n^{\nu}$ with $\nu$ choosen arbitrarily in $\mathbb{N}$.

\begin{theorem}
\label{th:intro_erreur_faible}
We recall that $T>0$ and $n \in \mathbb{N}^{\ast}$. Let $\nu >0$ and $\eta \geqslant 0$ and let us define $q_{\nu}=\max_{i\in \{1,\ldots,m(0,\nu) \}}(i\max(\beta,\kappa(1,q_i(\nu,0)))$. \\
 
 Assume that (\ref{hyp:transport_regularite_semigroup}), (\ref{hyp:transport_regularite_semigroup_approx}), (\ref{hyp:transport_regularite_semigroup_dual}) and (\ref{hyp:transport_regularite_semigroup_approx_dual}) hold and that the short time estimates $E_n(l,\alpha,\beta,P,Q)$ (see (\ref{hyp:erreur_tems_cours_fonction_test_reg})), and $E_n(l,\alpha,\beta,P,Q)^{\ast}$ (see (\ref{hyp:erreur_tems_cours_fonction_test_reg_dual})) hold for every $l \in \{1,\ldots,l(\nu,\alpha)\}$. Moreover, assume that $R_{q_{\nu},\eta }(Q^{\delta_{n}^{1}})$ and $R_{q_{\nu},\eta}(P)$ (see \ref{hyp:reg_forte}) hold. Then, for every $f \in \mathcal{M}_b(\mathbb{R}^d)$,
\begin{align*}
\Vert P_{T}f-\hat{Q}^{\nu,\delta_{n}^{0}}_{0,T}f \Vert_{\infty }\leqslant \frac{1}{n^{\nu}}\Vert f \Vert _{\infty }  \frac{C(1+ T^{p})}{T(\nu)^{\eta}}  .  
\end{align*}
with $T(\nu)=\inf \left\{t \in \pi^{\delta_{n}^{1}},t \geqslant T\frac{n-m(0,v)}{n(m(0,v)+1)} \right\}$ and $p$ which depends on $\nu$ and $\alpha$.
\end{theorem}

\begin{proof} 
We  proceed by recurrence. Notice that, using the Lindeberg decomposition,  $\hat{E}_n(l,(l+1)(\alpha+1)-1,\kappa,P,Q)$ (see (\ref{hyp:erreur_tems_cours_fonction_test_regschema_nu})) and $\hat{E}(l,(l+1)(\alpha+1)-1,\kappa,P,Q)^{\ast}$ (see (\ref{hyp:erreur_tems_cours_fonction_test_reg_schema_nu})) hold. for any $l \in \mathbb{N}$. Now, let us assume that $\hat{E}_n(1,\nu+1,\kappa,P,Q)$ and $\hat{E}(1,\nu+1,\kappa,P,Q)^{\ast}$ hold.
In order to prove this result, we introduce a reprensentation for the semigroup $(P_{t})_{t \geqslant 0}$ which relies on the family of semigroup $\left( \left( Q^{\delta_{n}^{l}}_{t} \right)_{t \in \pi^{\delta_{n}^{l}}} \right)_{l \in \mathbb{N}}$. In particular we have, for every $t \in\pi^{\delta_{n}^{l}}$,
\begin{align*}
P_{t,t+\delta_{n}^{l}}=Q^{\delta_{n}^{l+1}}_{t,t+\delta_{n}^{l}}+\sum_{i=1}^{m\left( l,\nu \right)-1} I^{\delta_{n}^{l+1}}_{t,t+\delta_{n}^{l},i}+ R^{\delta_{n}^{l+1}}_{t,t+\delta_{n}^{l},m\left( l,\nu \right)}
\end{align*}
with 
\begin{align*}
I^{\delta_{n}^{l+1}}_{t,t+\delta_{n}^{l},i} =& \sum_{t=t_0 <\ldots<t_{i}\leqslant t+\delta_{n}^{l} \in \pi^{\delta_{n}^{l+1}}}  \left(  \prod_{j=1}^{i}Q^{\delta_{n}^{l+1}}_{t_{j-1},t_{j}-\delta_{n}^{l+1}}  \left( P_{t_{j}-\delta_{n}^{l+1},t_{j}} - Q^{\delta_{n}^{l+1}}_{t_{j}-\delta_{n}^{l+1},t_{j}}   \right)   \right)  Q^{\delta_{n}^{l+1}}_{t_{i},t+\delta_{n}^{l}} .
\end{align*}
and
\begin{align*}
R^{\delta_{n}^{l+1}}_{t,t+\delta_{n}^{l},m} =& \sum_{t=t_0 <\ldots<t_{m}\leqslant t+\delta_{n}^{l} \in \pi^{\delta_{n}^{l+1}}} \left(   \prod_{j=1}^{m} Q^{\delta_{n}^{l+1}}_{t_{j-1},t_{j}-\delta_{n}^{l+1}} \left( P_{t_{j}-\delta_{n}^{l+1},t_{j}} - Q^{\delta_{n}^{l+1}}_{t_{j}-\delta_{n}^{l+1},t_{j}}   \right)   \right)  P_{t_{m},t+\delta_{n}^{l}}.
\end{align*}

It follows that
\begin{align*}
P_{T}f-\hat{Q}^{\nu,\delta_{n}^{0}}_{0,T}= \sum_{i=1}^{m\left(0,\nu \right)-1} I^{\delta_{n}^{1}}_{0,T,i}-\hat{I}^{\delta_{n}^{1}}_{0,T,i}+ R^{\delta_{n}^{1}}_{0,T,m\left(0,\nu \right)}
\end{align*}
with

\begin{align*}
I^{\delta_{n}^{1}}_{0,T,i}-\hat{I}^{\delta_{n}^{1}}_{0,T,i}  = & \sum_{0=t_0 <\ldots<t_{i}\leqslant T \in \pi^{\delta_{n}^{1}}} \left(    \prod_{j=1}^{i} Q^{\delta_{n}^{1}}_{t_{j-1},t_{j}-\delta_{n}^{1}}  \left( P_{t_{j}-\delta_{n}^{1},t_{j}} - Q^{\delta_{n}^{1}}_{t_{j}-\delta_{n}^{1},t_{j}}   \right)  \right) Q^{\delta_{n}^{1}}_{t_{i},T} \\
&  -\sum_{0=t_0 <\ldots<t_{i}\leqslant T \in \pi^{\delta_{n}^{1}}} \left(   \prod_{j=1}^{i} Q^{\delta_{n}^{1}}_{t_{j-1},t_{j}-\delta_{n}^{1}} \left( P_{t_{j}-\delta_{n}^{1},t_{j}}   -Q^{\delta_{n}^{1}}_{t_{j}-\delta_{n}^{1},t_{j}}  +\hat{Q}^{q_i\left( 0,\nu \right),\delta_{n}^{1}}_{t_{j}-\delta_{n}^{1},t_{j}}- P_{t_{j}-\delta_{n}^{1},t_{j}}  \right)   \right) Q^{\delta_{n}^{1}}_{t_{i},T} 
\end{align*}

More particularly, we can write
\begin{align*}
I^{\delta_{n}^{1}}_{0,T,i}-\hat{I}^{\delta_{n}^{1}}_{0,T,i}=& -\sum_{0=t_0 <\ldots<t_{i}\leqslant T \in \pi^{\delta_{n}^{1}}} \sum_{h=1}^{2^i} \left(  \prod_{j=1}^{i} Q^{\delta_{n}^{1}}_{t_{j-1},t_{j}-\delta_{n}^{1}} \Lambda^{\delta_{n}^{1},h}_{t_{j}-\delta_{n}^{1},t_{j}}    \right)  Q^{\delta_{n}^{1}}_{t_{i},T}
\end{align*}
with, for every $j \in \{1,\ldots,i\}$, $ \Lambda^{\delta_{n}^{1},h}_{t_{j}-\delta_{n}^{1},t_{j}} \in   \left\{ P_{t_{j}-\delta_{n}^{1},t_{j}}-Q^{\delta_{n}^{1}}_{t_{j}-\delta_{n}^{1},t_{j}}, \hat{Q}^{q_i\left( 0,\nu \right),\delta_{n}^{1}}_{t_{j}-\delta_{n}^{1},t_{j}}-P_{t_{j}-\delta_{n}^{1},t_{j}}   \right\}$. Moreover, we notice that the case $\Lambda^{\delta_{n}^{1},h}_{t_{j}-\delta_{n}^{1},t_{j}}=P_{t_{j}-\delta_{n}^{1},t_{j}}- Q^{\delta_{n}^{1}}_{t_{j}-\delta_{n}^{1},t_{j}}$ for every $j \in \{1,\ldots,i\}$ is excluded.
Using this decomposition, it is is sufficient to prove that 
\begin{align*}
\left \Vert  \left(  \prod_{j=1}^{i}  Q^{\delta_{n}^{1}}_{t_{j-1},t_{j}-\delta_{n}^{1}}  \Lambda^{\delta_{n}^{1},h}_{t_{j}-\delta_{n}^{1},t_{j}}   \right) Q^{\delta_{n}^{1}}_{t_{i},T}  f \right \Vert _{\infty
}\leqslant \frac{C(1+T^{\nu})}{%
T(\nu)^{\eta}}\Vert f\Vert _{\infty } / n^{\nu+i},
\end{align*}
and, for the remainder,
\begin{align*}
\left \Vert   \left( \prod_{j=1}^{m\left( l,\nu \right)} Q^{\delta_{n}^{1}}_{t_{j-1},t_{j}-\delta_{n}^{1}} \ \left( P_{t_{j}-\delta_{n}^{1},t_{j}} - Q^{\delta_{n}^{1}}_{t_{j}-\delta_{n}^{1},t_{j}}   \right) \right)P_{t_{m\left( l,\nu \right)},T}  f \right \Vert _{\infty
}\leqslant \frac{C(1+T^{\nu})}{%
T(\nu)^{\eta}}\Vert f \Vert _{\infty } / n^{\nu+i}.
\end{align*}
We focus on the study of $\sum_{i=1}^{m\left(0,\nu \right)-1} I^{\delta_{n}^{1}}_{0,T,i}-\hat{I}^{\delta_{n}^{1}}_{0,T,i}$. The study of $R^{\delta_{n}^{1}}_{0,T,m\left(0,\nu \right)}$ is similar so we leave it out. 
 First we notice that, using the convention $t_{i+1}=T+\delta_{n}^{1}$, for ${j_i}=\argsup_{j\in \{1,\ldots,i+1\}} \{ t_{j}-\delta_{n}^{1}-t_{j-1} \}$, we have $t_{j_i}-\delta_{n}^{1}-t_{j_i-1}\geqslant T(\nu) \coloneqq \inf \left\{t \in \pi^{\delta_{n}^{1}},t \geqslant T\frac{n-m(0,v)}{n(m(0,v)+1)} \right\}$.


Let $f \in \mathcal{C}^{\infty}_{c}(\mathbb{R}^{d}$. Using succesively $E_n(1,\alpha,\beta,P,Q)$ (see (\ref{hyp:erreur_tems_cours_fonction_test_reg})) or  $\hat{E}_n(1,q_i(\nu,0),\kappa,P,Q)$ (see (\ref{hyp:erreur_tems_cours_fonction_test_regschema_nu})) with (\ref{hyp:transport_regularite_semigroup_approx}), it follows that


\begin{align*}
& \left \Vert  \left( \prod_{j=1}^{i}  Q^{\delta_{n}^{1}}_{t_{j-1},t_{j}-\delta_{n}^{1}} \Lambda^{\delta_{n}^{1},h}_{t_{j}-\delta_{n}^{1},t_{j}} \right)    Q^{\delta_{n}^{1}}_{t_{i},T} f \right \Vert _{\infty}\\
&= \left \Vert  \left( \prod_{j=1}^{j_i-1}   Q^{\delta_{n}^{1}}_{t_{j-1},t_{j}-\delta_{n}^{1}} \Lambda^{\delta_{n}^{1},h}_{t_{j}-\delta_{n}^{1},t_{j}} \right) \left(  Q^{\delta_{n}^{1}}_{t_{j_i-1},t_{j_i}-\delta_{n}^{1}}   \Lambda^{\delta_{n}^{1},h}_{t_{j_i}-\delta_{n}^{1},t_{j_i}} \right) \left(    \prod_{j=j_i+1}^{i} Q^{\delta_{n}^{1}}_{t_{j-1},t_{j}-\delta_{n}^{1}} \Lambda^{\delta_{n}^{1},h}_{t_{j}-\delta_{n}^{1},t_{j}} \right)    Q^{\delta_{n}^{1}}_{t_{i},T}   f \right \Vert _{\infty
} \\ 
&\leqslant  C \left \Vert    \left(  Q^{\delta_{n}^{1}}_{t_{j_i-1},t_{j_i}-\delta_{n}^{1}}   \Lambda^{\delta_{n}^{1},h}_{t_{j_i}-\delta_{n}^{1},t_{j_i}} \right) \left(    \prod_{j=j_i+1}^{i} Q^{\delta_{n}^{1}}_{t_{j-1},t_{j}-\delta_{n}^{1}} \Lambda^{\delta_{n}^{1},h}_{t_{j}-\delta_{n}^{1},t_{j}} \right)    Q^{\delta_{n}^{1}}_{t_{i},T}    f \right \Vert _{(j_i-1) \max(\beta,\kappa(1,q_i(\nu,0))),\infty } \frac{1+T^{p(\nu,\alpha)}}{n^{\sum_{j=1}^{j_i-1} q_i^h(j)}},
\end{align*}

with $q_i^h(j)=q_i(0,\nu)$ if $ \Lambda^{\delta_{n}^{1},h}_{t_{j}-\delta_{n}^{1},t_{j}} =  \hat{Q}^{q_i\left( 0,\nu \right),\delta_{n}^{1}}_{t_{j}-\delta_{n}^{1},t_{j}} - P_{t_{j}-\delta_{n}^{1},t_{j}}$ and  $q_i^h(j)=\alpha+1$ if $ \Lambda^{\delta_{n}^{1},h}_{t_{j}-\delta_{n}^{1},t_{j}} =P_{t_{j}-\delta_{n}^{1},t_{j}}- Q^{\delta_{n}^{1}}_{t_{j}-\delta_{n}^{1},t_{j}}$ and $p(\nu,\alpha)$ is a constant depending on $\nu$ and $\alpha$ (and which may change value in the following lines or according to $i$). Notice that it is not possible to have $q_i^h(j) = \alpha+1$ for every $j \in \{1,\ldots,i\}$ and that there exists $\hat{j}_i \in \{1,\ldots,i\}$, such that
\begin{align*}
\sum_{j=1}^{i} q_i^h(j)=\hat{j}_i q_i(0,\nu)+(i-\hat{j}_i)(\alpha+1) \geqslant \nu+i ,
\end{align*}

Now, for $\varepsilon > 0$, we consider $\phi _{\varepsilon
}(x)=\varepsilon ^{-d}\phi (\varepsilon ^{-1}x)$ with $\phi \in
\mathcal{C}^{\infty}_{c}(\mathbb{R}^{d}),\; \phi \geqslant 0.$ and for a fixed $x_{0} \in \mathbb{R}^{d}$, we define $\phi
_{\varepsilon ,x_{0}}(x)=\phi _{\varepsilon }(x-x_{0}).$ Moreover, denote 
\begin{align*}
\Gamma_i=\Lambda^{\delta_{n}^{1},h}_{t_{j_i}-\delta_{n}^{1},t_{j_i}}  \left(    \prod_{j=j_i+1}^{i} Q^{\delta_{n}^{1}}_{t_{j-1},t_{j}-\delta_{n}^{1}} \Lambda^{\delta_{n}^{1},h}_{t_{j}-\delta_{n}^{1},t_{j}} \right)    Q^{\delta_{n}^{1}}_{t_{i},T}   
\end{align*}

 Since we have (\ref{hyp:transport_regularite_semigroup}) and (\ref{hyp:transport_regularite_semigroup_approx}),
$ Q^{\delta_{n}^{1}}_{t_{j-1},t_{j}-\delta_{n}^{1}}\Gamma_i   f $ belongs to $\mathcal{C}^{\infty}_{b} (\mathbb{R}^{d})$. Using succesively $E_n(1,\alpha,\beta,P,Q)^{\ast}$ (see (\ref{hyp:erreur_tems_cours_fonction_test_reg_dual})) or  $\hat{E}(1,q_i(\nu,0),\kappa,P,Q)^{\ast}$ (see (\ref{hyp:erreur_tems_cours_fonction_test_reg_schema_nu})) with (\ref{hyp:transport_regularite_semigroup_approx_dual}), it follows that for a multi-index $\gamma \in \mathbb{N}^{d}$, $x_0 \in \mathbb{R}^d$,

\begin{align*}
\vert  \partial_{x}^{\gamma}  Q^{\delta_{n}^{1}}_{t_{j-1},t_{j}-\delta_{n}^{1}}\Gamma_i    & f (x_{0})\vert =  \lim_{\varepsilon
\rightarrow 0}\vert \langle  \Gamma_i^{\ast} Q^{\delta_{n}^{1},\ast}_{t_{j-1},t_{j}-\delta_{n}^{1}} \partial_{x}^{\gamma} \phi_{ \varepsilon,x_{0}}, f \rangle \vert \\
\leqslant & C \sup_{\varepsilon >0} \left \Vert Q^{\delta_{n}^{1},\ast}_{t_{j_i-1},t_{j_i}-\delta_{n}^{1}} \partial_{x}^{\gamma} \phi_{ \varepsilon,x_{0}} \right\Vert _{ (i-j_i+1) \max(\beta,\kappa(1,q_i(\nu,0))),1}  \left\Vert f \right\Vert_{\infty}  \frac{1+T^{p(\nu,\alpha)}}{n^{\sum_{i=1}^{j_i-1} q_i^h(j)}}  \\
\end{align*}%
Our concern is the case $\vert \gamma \vert \leqslant (j_i-1) \max(\beta,\kappa(1,q_i(\nu,0)))$. Using $R_{q_{\nu},\eta }(Q^{\delta_{n}^{1}})$ (see (\ref{hyp:reg_forte})) and more particularly the implication (\ref{hyp:reg_forte_dual}) on $Q^{\delta_{n}^{1},\ast}$, it follows that,

\begin{align*}
\left \Vert  \left( \prod_{j=1}^{i}  Q^{\delta_{n}^{1}}_{t_{j-1},t_{j}-\delta_{n}^{1}} \Lambda^{\delta_{n}^{1},h}_{t_{j}-\delta_{n}^{1},t_{j}} \right)    Q^{\delta_{n}^{1}}_{t_{i},T} f \right \Vert _{\infty} \leqslant  C \sup_{\varepsilon >0} \frac{1}{T(\nu)^{\eta}} \left\Vert \phi _{\varepsilon ,x_{0}} \right\Vert _{1}  \left\Vert f \right\Vert_{\infty}  \frac{1+T^{p(\nu,\alpha)}}{n^{\sum_{j=1}^{i} q_i^h(j)}}  
\end{align*}

and since $\sum_{j=1}^{i} q_i^h(j) \geqslant \nu + i$ and $\Vert \phi _{\varepsilon ,x_{0}}\Vert
_{1}=\Vert \phi \Vert _{1}\leqslant C,$ the proof is completed as long as $f \in \mathcal{C}_{c}(\mathbb{R}^{d})$. The extension to $f \in \mathcal{M}_{b}(\mathbb{R}^{d})$ is guaranteed by the Lusin's theorem.

\end{proof}

We are now interested by giving a variant of Theorem in which the regularization hypothesis is not required for $P$ or $Q$ but for some modifications of those semigroups.
\begin{proposition}
\label{prop:intro_erreur_faible_modif}
We recall that $T>0$ and $n \in \mathbb{N}^{\ast}$. Let $\nu >0$ and $\eta \geqslant 0$ and let us define $q_{\nu}=\max_{i\in \{1,\ldots,m(0,\nu) \}}(i\max(\beta,\kappa(1,q_i(\nu,0)))$. \\

We assume that (\ref{hyp:transport_regularite_semigroup}) and (\ref{hyp:transport_regularite_semigroup_approx}) and (\ref{hyp:transport_regularite_semigroup_dual}), (\ref{hyp:transport_regularite_semigroup_approx_dual}) hold and that the short time estimates $E_n(l,\alpha,\beta,P,Q)$ (see (\ref{hyp:erreur_tems_cours_fonction_test_reg})), and $E_n(l,\alpha,\beta,P,Q)^{\ast}$ (see (\ref{hyp:erreur_tems_cours_fonction_test_reg_dual})) hold for every $l \in \{1,\ldots,l(\nu,\alpha)\}$. Also, assume that there exists a modification $\overline{Q}^{\delta_{n}^{1}}$ (respectively $\overline{P}$) of $Q^{\delta_{n}^{1}}$ (resp. $P$) which satisfy  $R_{q_{\nu},\eta }(\overline{Q}^{\delta_{n}^{1}})$ (resp. $R_{q_{\nu},\eta }(\overline{P})$) (see \ref{hyp:reg_forte}) and such that for every $f \in \mathcal{M}_b(\mathbb{R}^d)$,
 
 \begin{align}
\forall s,t \in \pi^{\delta_{n}^{1}}, s < t  , \qquad \Vert Q^{\delta_{n}^{1}}_{s,t}f- \overline{Q}^{\delta_{n}^{1}}_{s,t}f\Vert _{\infty }+\Vert
P_{s,t}f-\overline{P}_{s,t}f\Vert _{\infty }\leqslant \frac{1+T^{\overline{p}}}{n^{\nu+m(0,\nu)}} \Vert f\Vert _{\infty } C(t-s)^{-\eta }.  
 \label{hyp:approx_distance_total_variation_loc}
\end{align}
where $\overline{p}$ depends on $\nu$ and $\alpha$
 Then, for every $f \in \mathcal{M}_b(\mathbb{R}^d)$,
\begin{align*}
\Vert P_{T}f-\hat{Q}^{\nu,\delta_{n}^{0}}_{0,T}f \Vert_{\infty }\leqslant  \frac{1}{n^{\nu}}  \Vert f \Vert _{\infty }  \frac{C(1+T^{p})}{T(\nu)^{\eta }}.  
\end{align*}
with $T(\nu)=\inf \left\{t \in \pi^{\delta_{n}^{1}},t \geqslant T\frac{n-m(0,v)}{n(m(0,v)+1)} \right\}$ and $p$ which depends on $\nu$ and $\alpha$.
\end{proposition}

\begin{remark}
Notice that $\overline{P}$ and $\overline{Q}^{\delta_{n}^{1}}$ are not supposed
to satisfy the semigroup property.
\end{remark}

\begin{proof}
 The proof follows the same line as the one of the previous
Theorem \ref{th:intro_erreur_faible}. Consequently, we only focus on the specificity of this proof, avoiding arguments which are similar to the previous proof. In particular we study, for every $i \in \{1,\ldots,m(0,\nu)\}$,

\begin{align*}
\Vert  Q^{\delta_{n}^{1}}_{t_{i},T} & \prod_{j=1}^{i} \Lambda^{\delta_{n}^{1},h}_{t_{j}-\delta_{n}^{1},t_{j}}   Q^{\delta_{n}^{1}}_{t_{j-1},t_{j}-\delta_{n}^{1}}  f\Vert _{\infty
}\\
\leqslant &  \Vert  Q^{\delta_{n}^{1}}_{t_{i},T}   \prod_{j=j_i+1}^{i} \Lambda^{\delta_{n}^{1},h}_{t_{j}-\delta_{n}^{1},t_{j}}   Q^{\delta_{n}^{1}}_{t_{j-1},t_{j}-\delta_{n}^{1}} \Lambda^{\delta_{n}^{1},h}_{t_{j_i}-\delta_{n}^{1},t_{j_i}}  \overline{Q}^{\delta_{n}^{1}}_{t_{j_i-1},t_{j_i}-\delta_{n}^{1}} \prod_{j=1}^{j_i-1} \Lambda^{\delta_{n}^{1},h}_{t_{j}-\delta_{n}^{1},t_{j}}   Q^{\delta_{n}^{1}}_{t_{j-1},t_{j}-\delta_{n}^{1}}f\Vert _{\infty
} \\ 
& +   \Vert  Q^{\delta_{n}^{1}}_{t_{i},T}   \prod_{j=j_i+1}^{i} \Lambda^{\delta_{n}^{1},h}_{t_{j}-\delta_{n}^{1},t_{j}}   Q^{\delta_{n}^{1}}_{t_{j-1},t_{j}-\delta_{n}^{1}} \Lambda^{\delta_{n}^{1},h}_{t_{j_i}-\delta_{n}^{1},t_{j_i}} \left(Q^{\delta_{n}^{1}}_{t_{j_i-1},t_{j_i}-\delta_{n}^{1}}-  \overline{Q}^{\delta_{n}^{1}}_{t_{j_i-1},t_{j_i}-\delta_{n}^{1}} \right) \prod_{j=1}^{j_i-1} \Lambda^{\delta_{n}^{1},h}_{t_{j}-\delta_{n}^{1},t_{j}}   Q^{\delta_{n}^{1}}_{t_{j-1},t_{j}-\delta_{n}^{1}}f\Vert _{\infty
} .
\end{align*}
The first term is studied similarly as in Theorem \ref{th:intro_erreur_faible} We use the same notations as introduced in this proof. For the second term we use (\ref{hyp:approx_distance_total_variation_loc}) together with successive application of (\ref{hyp:transport_regularite_semigroup_approx}) and it follows that

\begin{align*}
\Vert  Q^{\delta_{n}^{1}}_{t_{i},T}  \prod_{j=1}^{i} \Lambda^{\delta_{n}^{1},h}_{t_{j}-\delta_{n}^{1},t_{j}}   Q^{\delta_{n}^{1}}_{t_{j-1},t_{j}-\delta_{n}^{1}}  f\Vert _{\infty
} \leqslant    \frac{C(1+ T^{p(\nu,\alpha)}) }{T(\nu)^{\eta}} \Vert f \Vert _{\infty }  \frac{1}{n^{\nu+i}}+ \frac{C (1+T^{\overline{p}})}{T(\nu)^{\eta}} \Vert f \Vert _{\infty }  \frac{1}{n^{\nu+m(0,\nu)}}
\end{align*}
Notice that the study of the remainder $R^{\delta_{n}^{1}}_{0,T,m\left(0,\nu \right)}$ which appears in the proof of
Proposition \ref{th:intro_erreur_faible} is similar. Rearranging the terms completes the proof.
\end{proof}

At this point, we establish a total variation convergence result which does not require that the regularization property hold for $P$ but only on the collection of semigroups $\left(Q^{\delta}\right)_{\delta>0}$. More specifically, we consider the following hypothesis : Recall that  $q_{\nu}=\max_{i\in \{1,\ldots,m(0,\nu) \}}(i\max(\beta,\kappa(1,q_i(\nu,0)))$ with the definition of $m$, $\kappa$ and $q_i$ given in \ref{hyp:erreur_tems_cours_fonction_test_regschema_nu} and that $l(\nu,\alpha)=\lceil  \nu / \alpha  \rceil$. Let us consider the hypothesis:\\

\begin{center}
$\overline{R}_{n,\nu,\eta }(Q)$ \\
$\equiv$ \\
For every $k \in \mathbb{N}^{\ast}$,
\end{center}

  \begin{center}
    \begin{minipage}{0.89\textwidth}
\begin{enumerate}[label=$\overline{R}_{n,\nu,\eta }(Q)$\textbf{.\roman*.}]
\item \label{reg_property_mes_cont_cauchy_1}  (\ref{hyp:transport_regularite_semigroup_approx}) and (\ref{hyp:transport_regularite_semigroup_approx_dual}) hold with $l$ replaced by $k$.
\item \label{reg_property_mes_cont_cauchy_3} There exists a modification $\overline{Q}^{\delta_{n}^{k}}$ of $Q^{\delta_{n}^{k}}$ which satisfies  $R_{q_{\nu},\eta }(\overline{Q}^{\delta_{n}^{k}})$  (see (\ref{hyp:reg_forte})) and such that: $\forall s,t \in  \pi^{\delta_{k}^{1}}, s < t$,
 \begin{align}
\Vert Q^{\delta_{n}^{k}}_{s,t}f- \overline{Q}^{\delta_{n}^{k}}_{s,t}f\Vert _{\infty }\leqslant C(t-s)^{-\eta} \frac{\Vert f\Vert _{\infty }}{n^{\nu+m(0,\nu)}}.
 \label{hyp:approx_distance_total_variation_loc_theorem}
\end{align}
\end{enumerate}
\end{minipage}
\end{center}

\begin{theorem}
\label{theo:distance_density} 
We recall that $T>0$ and $n \in \mathbb{N}^{\ast}$. Let $\nu >0$ and $\eta \geqslant 0$.  \\
 
 Assume that (\ref{hyp:transport_regularite_semigroup}) and (\ref{hyp:transport_regularite_semigroup_dual}) hold. Assume that $\overline{R}_{n,\nu,\eta }(Q)$ hold and that for every $k \in \mathbb{N}^{\ast}$, the short time estimates $E_n(k,\alpha,\beta,P,Q)$ (see (\ref{hyp:erreur_tems_cours_fonction_test_reg})), and $E_n(k,\alpha,\beta,P,Q)^{\ast}$ (see (\ref{hyp:erreur_tems_cours_fonction_test_reg_dual})) hold. Then, for every $f \in \mathcal{M}_b(\mathbb{R}^d)$,
\begin{align}
\Vert P_{T}f-\hat{Q}^{\nu,\delta_{n}^{0}}_{0,T}f \Vert_{\infty }\leqslant  \frac{1}{n^{\nu}}  \Vert f \Vert _{\infty }    \frac{C (1+T^{p}) }{T(\nu)^{\eta }}.
\label{eq:theo_approx_semigroup_theorem}
\end{align}
with $T(\nu)=\inf \left\{t \in \pi^{\delta_{n}^{1}},t \geqslant T\frac{n-m(0,v)}{n(m(0,v)+1)} \right\}$ and $p$ which depends on $\nu$ and $\alpha$.

\end{theorem}

\begin{remark}
The inequality (\ref{eq:theo_approx_semigroup_theorem}) is essentially a consequence of Theorem \ref{th:intro_erreur_faible}. However, we may not use directly this result, because we do not assume that the semigroup $P$ has the regularization property $R_{q_{\nu},\eta }(P)$ (see (\ref{hyp:reg_forte})) This is a result of main interest since we have to check the regularization properties for the approximations $Q^{\delta}$ only. Notice that the method we use does not allow to prove the same result when assuming regularization hypothesis on $P$ instead of $Q$. The reason is that our proof consist in considering $P$ as the limit of $Q^{\delta}$ as $\delta$ tends to 0. It is not possible to act similarly in the other way as $P$ does not depend on such a ${\delta}$.
\end{remark}


\begin{proof}[Proof of Theorem \ref{theo:distance_density}]

We fix $n \in \mathbb{N}^{\ast}$ and we study the sequence of discrete semigroups $\left( \left( Q^{\delta_{n}^{k}}_{s,t} \right)_{s,t \in \pi^{\delta_{n}^{k}};s \leqslant t} \right)_{k \in \mathbb{N}^{\ast}}$.

\textbf{Step 1.} We show that for every bounded and measurable test function $f$, $\left(Q^{\delta_{n}^{k}}_{0,T}f \right)_{k \in \mathbb{N}^{\ast}}$ is Cauchy in $( \mathcal{M}_{b}(\mathbb{R}^{d}),\Vert.\Vert_{\infty})$ and that, for $k \geqslant l(\nu,\alpha)$,
\begin{align}
\label{eq:estim_tv_cauchy_limit}
\left \Vert P_{T}f-Q^{\delta_{n}^{k}}_{0,T}f \right \Vert_{\infty} \leqslant \frac{CT^{\alpha+1}}{T(\nu)^{\eta}} \left \Vert f \right \Vert_{\infty} \frac{1}{n^{\nu}}.
\end{align}
where $p(\nu,\alpha)$ is a constant depending on $\nu$ and $\alpha$ (and which may change from line to line in the following calculus).\\

For $k' \geqslant k \in \mathbb{N}^{\ast}$, following the Lindeberg decomposition yields 

\begin{align*}
\left \Vert Q^{\delta_{n}^{k'}}_{0,T}f-Q^{\delta_{n}^{k}}_{0,T}f \right \Vert_{\infty} \leqslant   \sum_{m=1}^{n} \left \Vert   Q^{\delta_{n}^{k}}_{0,\left(m-1\right)\delta_{n}^{1}} \left( Q^{\delta_{n}^{k'}}_{\left(m-1\right)\delta_{n}^{1},m\delta_{n}^{1}} - Q^{\delta_{n}^{k}}_{\left(m-1\right)\delta_{n}^{1},m\delta_{n}^{1}}   \right)  Q^{\delta_{n}^{k'}}_{m\delta_{n}^{1},T}  f \right \Vert_{\infty}.
\end{align*}

Now notice that for $g \in \mathcal{C}_b^{\beta}\left(\mathbb{R}^d\right)$,
\begin{align*}
\left \Vert Q^{\delta_{n}^{k'}}_{\left(m-1\right)\delta_{n}^{1},m\delta_{n}^{1}}g-Q^{\delta_{n}^{k}}_{\left(m-1\right)\delta_{n}^{1},m\delta_{n}^{1}}g \right \Vert_{\infty}  \leqslant & \left \Vert P_{\left(m-1\right)\delta_{n}^{1},m\delta_{n}^{1}}g-Q^{\delta_{n}^{k'}}_{\left(m-1\right)\delta_{n}^{1},m\delta_{n}^{1}}g \right \Vert_{\infty} \\
& +\left \Vert P_{\left(m-1\right)\delta_{n}^{1},m\delta_{n}^{1}}g-Q^{\delta_{n}^{k}}_{\left(m-1\right)\delta_{n}^{1},m\delta_{n}^{1}}g \right \Vert_{\infty}
\end{align*}
with, as a consequence of $E_{n}(1,\alpha,\beta,P,Q)$ (see (\ref{hyp:erreur_tems_cours_fonction_test_reg})), $\left \Vert P_{\left(m-1\right)\delta_{n}^{1},m\delta_{n}^{1}}g-Q^{\delta_{n}^{k}}_{\left(m-1\right)\delta_{n}^{1},m\delta_{n}^{1}}g \right \Vert_{\infty} \leqslant C\left \Vert g \right \Vert_{\infty,\beta} \frac{T^{\alpha+1}}{n^{\alpha+1}} $ if $k=1$, and if $k>1$
\begin{align*}
\big \Vert P_{\left(m-1\right)\delta_{n}^{1},m\delta_{n}^{1}}f- &   Q^{\delta_{n}^{k}}_{\left(m-1\right)\delta_{n}^{1},m\delta_{n}^{1}}g \big \Vert_{\infty} \\
 & \leqslant   \sum_{u=1+n^{k-1} \left(m-1\right)}^{mn^{k-1}} \left \Vert  Q^{\delta_{n}^{k}}_{\left(m-1\right)\delta_{n}^{1},\left(u-1\right)\delta_{n}^{k}} \left( P_{\left(u-1\right)\delta_{n}^{k},u\delta_{n}^{k}} - Q^{\delta_{n}^{k}}_{\left(u-1\right)\delta_{n}^{k},u\delta_{n}^{k}}   \right) P_{u \delta_{n}^{k},m\delta_{n}^{1}}   g \right \Vert_{\infty} \\
& \leqslant C  \left \Vert g \right \Vert_{\infty,\beta} \frac{T^{\alpha+1}}{n^{k\alpha+1}}
\end{align*}

where we have used $E_{n}(k,\alpha,\beta,P,Q)$ (see (\ref{hyp:erreur_tems_cours_fonction_test_reg})). Consequently
\begin{align*}
\left \Vert Q^{\delta_{n}^{k'}}_{\left(m-1\right)\delta_{n}^{1},m\delta_{n}^{1}}g-Q^{\delta_{n}^{k}}_{\left(m-1\right)\delta_{n}^{1},m\delta_{n}^{1}}g \right \Vert_{\infty} \leqslant C  \left \Vert g \right \Vert_{\infty,\beta} \frac{T^{\alpha+1}}{n^{k\alpha+1}}
\end{align*}

 In the same way we deduce from $E_{n}(k,\alpha,\beta,P,Q)^{\ast}$ (see (\ref{hyp:erreur_tems_cours_fonction_test_reg_dual})) that

\begin{align*}
\left \Vert Q^{\delta_{n}^{k'},\ast}_{\left(m-1\right)\delta_{n}^{1},m\delta_{n}^{1}}g-Q^{\delta_{n}^{k},\ast}_{\left(m-1\right)\delta_{n}^{1},m\delta_{n}^{1}}g \right \Vert_{1} \leqslant C  \left \Vert g \right \Vert_{1,\beta} \frac{T^{\alpha+1}}{n^{k\alpha+1}}
\end{align*}

Combining those estimates with $R_{q_{\nu},\eta }(\overline{Q}^{\delta_{n}^{k}})$ and $R_{q_{\nu},\eta }(\overline{Q}^{\delta_{n}^{k'}})$ together with (\ref{hyp:approx_distance_total_variation_loc_theorem}), the same approach as in the proof of Proposition \ref{prop:intro_erreur_faible_modif} yields, for every $f \in \mathcal{M}_{b}(\mathbb{R}^{d})$,

\begin{align}
\label{eq:estim_tv_cauchy}
\left \Vert Q^{\delta_{n}^{k'}}_{0,T}f-Q^{\delta_{n}^{k}}_{0,T}f \right \Vert_{\infty} \leqslant  \frac{1}{n^{k\alpha}}\Vert f \Vert_{\infty} \frac{C(1+T^{p(k,\alpha)})}{T(\nu)^{\eta}}
\end{align}

The sequence $\left(Q^{\delta_{n}^{k}}_{0,T}f \right)_{k \in \mathbb{N}^{\ast}}$ is thus Cauchy in $( \mathcal{M}_{b}(\mathbb{R}^{d}),\Vert.\Vert_{\infty})$ and then $ \lim_{k \to \infty}Q^{\delta_{n}^{k}}_{0,T}f$ exists and belongs to $ \mathcal{M}_{b}(\mathbb{R}^{d})$. Moreover, remember that as soon as $f \in \mathcal{C}^{\infty}_c\left(\mathbb{R}^d\right)$, (\ref{hyp:erreur_tems_cours_fonction_test_regschema_nu}) holds and then $Q^{\delta_{n}^{k}}_{0,T}f $ converges to $P_Tf$ as $k$ tends to infinity so that it is also the case when $f$ is simply bounded and measurable. Taking $k \geqslant \frac{\nu}{\alpha}$ in (\ref{eq:estim_tv_cauchy}) and letting $k'$ tends to infinity, if follows that (\ref{eq:estim_tv_cauchy_limit}) holds.\\

\textbf{Step 2.} We now show that for every $k \geqslant l(\nu/\alpha)=\lceil \nu/\alpha \rceil$ and every $f \in \mathcal{M}_b$, 

\begin{align}
\label{eq;erreur_faible_cauchy}
\Vert  Q^{\delta_{n}^{k}}_{0,T}f - \hat{Q}^{\nu,\delta_{n}^{0}}_{0,T}f  \Vert_{\infty} \leqslant \frac{C(1+T^{p(\nu,\alpha)}) }{T(\nu)^{\eta}} \Vert f \Vert _{\infty }  \frac{1}{n^{\nu}}.
\end{align}

Let $k \geqslant l(\nu,\alpha)$. We remark that if we replace  $P$ by $Q^{\delta_{n}^{k}}$, the short time estimates $E_n(l,\alpha,\beta,Q^{\delta_{n}^{k}},Q)$ (see (\ref{hyp:erreur_tems_cours_fonction_test_reg})), and $E_n(l,\alpha,\beta,Q^{\delta_{n}^{k}},Q)^{\ast}$ (see (\ref{hyp:erreur_tems_cours_fonction_test_reg_dual})) still hold for every  $l \in \{1,\ldots,l(\nu,\alpha)\}$. \\

Moreover, from \ref{reg_property_mes_cont_cauchy_3}, for every $k \in \mathbb{N}^{\ast}$, the property $R_{q_{\nu},\eta}(\overline{Q}^{\delta_{n}^{k}})$ (see \ref{hyp:reg_forte}) holds for a modification $\overline{Q}^{\delta_{n}^{k}}$ of $Q^{\delta_{n}^{k}}$ which satisfies (\ref{hyp:approx_distance_total_variation_loc_theorem}). Therefore, all the assumption of Proposition \ref{th:intro_erreur_faible} are fulfilled when we replace $P$ by $Q^{\delta_{n}^{k}}$, so that $\ref{eq;erreur_faible_cauchy}$ holds.

\textbf{Step 3.} We combine (\ref{eq:estim_tv_cauchy_limit}) and (\ref{eq;erreur_faible_cauchy}) and (\ref{eq:theo_approx_semigroup_theorem}) follows.

\end{proof}

\section{Total variation convergence for a class of semigroups}
\subsection{A Class of Markov Semigroups} \;

In this section we investigate the regularization properties of $Q^{\delta_{n}^{1}}$ and $Q^{\delta_{n}^{l(\nu,\alpha)}}$ which are crucial to derive total variation convergence results through Theorem \ref{theo:distance_density}. In particular we propose an application where $Q^{\delta_{n}^{1}}$ and $Q^{\delta_{n}^{l(\nu,\alpha)}}$ are the discrete semigroups of discrete Markov processes defined through an abstract random recurrence. Regularization properties are then obtained for some modifications of those semigroups under H\"ormander assumptions. More particularly, we will obtain regularization property for modifications of the family of discrete semigroups $(Q^{\delta})_{\delta>0}$. Our approach is similar to the one developped in \cite{Rey_2024} where regularization properties were established for such semigroups in a slightly more general setting. We introudce the result from this paper we need and adapt to our current framework.\\

\noindent \textbf{Definition of the semigroups.}
We work on a probability space $(\Omega,\mathcal{F},\mathbb{P})$. For $\delta \in (0,1]$ and $N\in \mathbb{N}^{\ast}$, we consider a sequence of
independent random variables $Z^{\delta}_{t}\in \mathbb{R}^{N},\; t\in
\pi^{\delta,\ast}$, and we assume that $Z^{\delta}_{t}$, are centered with $ \mathbb{E}[ Z^{\delta,i}_{t} Z^{\delta,j}_{t}]=\mathbf{1}_{i,j}$ for every $i,j \in \mathbf{N} :=   \{1,\ldots,N\}$ and every $t\in
\pi^{\delta,\ast}$. We construct the $\mathbb{R}^{d}$-valued Markov process $(X^{\delta}_{t})_{t \in \pi^{\delta}}$ in the following way:%
\begin{align}
X^{\delta}_{t+\delta}=\psi(X^{\delta}_{t},t,\delta^{\frac{1}{2}} Z^{\delta}_{t+\delta}, \delta) , \quad t \in  \pi^{\delta}, \quad X^{\delta}_{0}=\mbox{\textsc{x}}^{\delta}_{0}\in \mathbb{R}^d \label{eq:schema_general}
\end{align}%

where
\begin{align*}
\psi \in \mathcal{C}^{\infty }( \mathbb{R}^{d}\times \mathbb{R}_+\times \mathbb{R}^{N} \times [0,1];\mathbb{R}^{d})\quad \mbox{and} \quad  \forall (x,t) \in \mathbb{R}^{d} \times \pi^{\delta},\psi
(x,t,0,0)=x.  
\end{align*}
Let us now define the discrete time semigroup associated to $(X^{\delta}_{t})_{t \in \pi^{\delta}}$. For every measurable function $f$ from $\mathbb{R}^{d}$ to $\mathbb{R}$, and every $x \in \mathbb{R}^{d}$,
\begin{align*}
\forall s,t \in \pi^{\delta}, s \leqslant t, \qquad  Q^{\delta}_{s,t}f(x) =\int_{\mathbb{R}^{d}} f(y) Q^{\delta}_{s,t}(x,\mbox{d} y)  :=  \mathbb{E}[f(X^{\delta}_{t}) \vert X^{\delta}_{s}=x] .
  \end{align*}
  
We will obtain regularization properties for modifications of this discrete semigroup.  Our approach relies on some hypothesis on $\psi$ and $Z^{\delta}$ we now present.  \\

\noindent \textbf{Hypothesis on $\psi$.  Boundaries and H\"ormander property.} 

We first consider a boundary assumption concerning the derivatives of $\psi$: For $r \in \mathbb{N}^{\ast}$, 
\begin{enumerate}[label=$\mathbf{A}_{1}^{\delta}(r)$.]
\item \label{Hypothese_pol_growth} 
There exists $\mathfrak{D}_{r}\geqslant 1,\mathfrak{p}_{r}  \in \mathbb{N}$ such that 
 for every $(x,t,z,y) \in \mathbb{R}^{d} \times \mathbb{R}_{+} \times \mathbb{R}^{N} \times [0,1]$,
\begin{align}
 \sum_{\vert \gamma^{x}
\vert + \vert \gamma^{t} \vert =0}^{r} \sum_{\vert \gamma^{z} \vert  +\vert \gamma^{y} \vert 
=1}^{r-\vert \gamma^{x} \vert - \vert \gamma^{t} \vert}\vert \partial
_{x}^{\gamma^{x}}\partial_{t}^{\gamma^{t}} \partial_{z}^{\gamma^{z}} \partial_{y}^{\gamma^{y}}\psi\vert_{\mathbb{R}^{d}} (x,t,z,y) \leqslant \mathfrak{D}_{r}(1+ \delta^{-\frac{\mathfrak{p}_{r}}{2}} \vert z \vert_{\mathbb{R}^{N}}^{\mathfrak{p}_{r}}),
\label{eq:hyp_1_Norme_adhoc_fonction_schema}
\end{align}

\end{enumerate}
Without loss of generality, we assume that the sequences $(\mathfrak{D}_{r})_{r \in \mathbb{N}^{\ast}}$ and $(\mathfrak{p}_{r})_{r \in \mathbb{N}^{\ast}}$ are non decreasing.  We denote $\mathbf{A}_{1}^{\delta}(+\infty)$ when $\mathbf{A}_{1}^{\delta}(r)$ is satisfied for every $r \in \mathbb{N}^{\ast}$.

The second hypothesis we need on $\psi$ is  uniform weak H\"ormander property on some vector fields we now introduce. We denote the Lie bracket of two $\mathcal{C}^1$ vector fields in $\mathbb{R}^{d}$, $[,]:(\mathcal{C}^1(\mathbb{R}^{d},\mathbb{R}^{d}))^{2} \to \mathcal{C}^0(\mathbb{R}^{d},\mathbb{R}^{d})$, $f_1,f_2 \mapsto [f_{1},f_{2}] :=   \nabla_{x}f_{2}f_{1}-\nabla_{x}f_{1}f_{2}$. \\
We denote $\tilde{V}_{0}= \partial_{y} \psi(.,.,0,0)$, $V_{0}  :=  \tilde{V}_{0}-\frac{1}{2}\sum_{i=1}^{N} \partial_{z^{i}}^{2} \psi(.,.,0,0)$, $V_{i} =\partial_{z^{i}} \psi(.,.,0,0)$, $i \in \mathbf{N}$, $\bar{V}_0=V_{0}-\frac{1}{2}\sum_{i=1}^{N} \nabla_{x} V_{i} V_{i}$. For a multi-index $\alpha \in \{0,\ldots,N\}^{\Vert \alpha \Vert}$ and $V:\mathbb{R}^{d} \times \mathbb{R}_{+} \to \mathbb{R}^{d}$, we define also $V^{[\alpha]}$ using the recurrence relation $V^{[(\alpha,0)]}=[\bar{V}_0,V^{[\alpha]}]+\partial_{t}V^{[\alpha]}+\frac{1}{2}\sum_{i=1}^N[ V_{i},[V_{i},V^{[\alpha]}]]$ and $V^{[(\alpha,j)]} :=   [V_{j},V^{[\alpha]}]$ if $j \in \{1,\ldots,N\}$ with the convention $V^{[\emptyset]}=V$.  We are now in a position to introduce our H\"ormander hypothesis on $\psi$: For $L \in \mathbb{N}$,  the order of our H\"ormander condition,  let us define for every $(x,t) \in \mathbb{R}^{d} \times \mathbb{R}_{+}$,

\begin{align}
\label{def:hormander_func}
\mathcal{V}_{L}(x,t)  :=    1 \wedge \inf_{\mathbf{b} \in \mathbb{R}^{d}, \vert \mathbf{b} \vert_{\mathbb{R}^{d}} =1}  \sum_{\underset{ \Vert \alpha \Vert \leqslant L}{\alpha \in \{0,\ldots,N\}^{\Vert \alpha \Vert};}} \sum_{i=1}^{N} \langle V^{[\alpha]}_{i}(x,t) , \mathbf{b} \rangle_{\mathbb{R}^{d}}^{2} .
\end{align}

We introduce:
\begin{enumerate}[label=$\mathbf{A}_{2}^{\infty}(L)$.]
\item \label{Hypothese_hormander}  Our uniform weak H\"ormander property of order $L$,
\begin{align}
\label{hyp:unif_hormander}
\mathcal{V}_{L}^{\infty}   :=  \inf_{t \in \mathbb{R}_{+}}\inf_{x \in \mathbb{R}^{d}}\mathcal{V}_{L}(x,t)>0 .
\end{align}.
When $L=0$ this hypothesis is also called uniform ellipticity.
\end{enumerate}

\begin{remark}
 The reason we refer to the denomination H\"ormander resides in the fact (\ref{hyp:unif_hormander}) exactly corresponds to the commonly known uniform weak H\"ormander property the solution to the SDE (\ref{eq:SDE_inv_strato}), written in Stratonovich form as

\begin{align}
\label{eq:SDE_inv_strato}
X_{t}=X_{0}+\int_{0}^{t} \bar{V}_{0}(X_{s},s) \mbox{d}s+ \sum_{i=1}^{N} \int_{0}^{t} V_{i}(X_{s},s) \circ \mbox{d}W^{i}_{s}, \quad t \geqslant 0,X_{0}\in \mathbb{R}^{d}
\end{align}
where $((W^{i}_{t})_{t \geqslant 0}, i \in  \{1,\ldots,N\})$ are $N$ independent $\mathbb{R}$-valued standard Brownian motions and $\circ  \mbox{d}W^{i}_{s}$ stands for the Stratonovich integral $w.r.t.$ $(W^{i}_{t})_{t \geqslant 0}$. 
 When $L=0$, \ref{hyp:unif_hormander}) is also called uniform elliptic property We points out that (\ref{hyp:unif_hormander}) is equivalent to assume that $\mathbb{R}^{d}$ is spanned by $\{V_{i}^{[\alpha]}(x,t), i \in \{1,\ldots,N\},\alpha \in \{0,\ldots,N\}^{\Vert \alpha \Vert},  \Vert \alpha \Vert \leqslant L \}$ for every $(x,t) \in \mathbb{R}^{d} \times \mathbb{R}_{+}$.  In addition, the total variation type result we are going to establish consists in showing that $\hat{Q}^{\nu,\delta_{n}^{0}}_{0,T}f$ converges to $P_{T}f(x) := \mathbb{E}[f(X_{t}) \vert X_{0}=x]$ for any bounded and measurable test function $f$.
\end{remark}

\noindent \textbf{Hypothesis on $Z^{\delta}$.  Lebesgue lower bounded distributions.}  A first assumption concerns the finiteness of the moment of $Z^{\delta}$: For $p \geqslant 0$,
\begin{enumerate}[label=$\mathbf{A}_{3}^{\delta}(p)$.]
\item \label{hyp:moment_borne_Z_assumption} 
\begin{align}
\label{eq:hyp:moment_borne_Z} 
\mathfrak{M}_{p}(Z^{\delta}) :=   1\vee \sup_{t \in  \pi^{\delta,\ast}}\mathbb{E}[\vert Z^{\delta}_{t}\vert_{\mathbb{R}^{N}} ^{p}]<\infty .
\end{align}.
\end{enumerate}We denote $\mathbf{A}_{3}^{\delta}(+\infty)$ the assumption such that $\mathbf{A}_{3}^{\delta}(p)$ is satisfied for every $p \geqslant 0$. \\

A second assumption is made on the distribution of $Z^{\delta}$.  We suppose that the distribution of $Z^{\delta}$ is Lebesgue lower bounded:
\begin{enumerate}[label=$\mathbf{A}_{4}^{\delta}$.]
\item \label{hyp:moment_borne_Z} 
There exists $z_{\ast}=(z_{\ast ,t})_{t \in \pi^{\delta,\ast} }$ taking its values in $\mathbb{R}^{N}$ and $%
\varepsilon _{\ast },r_{\ast }>0$ such that for every Borel set $A\subset
\mathbb{R}^{N}$ and every $t \in \pi^{\delta,\ast},$%
\begin{align}
L^{\delta}_{z_{\ast }}(\varepsilon _{\ast },r_{\ast })\qquad \mathbb{P}(Z^{\delta}_{t}\in A)\geqslant
\varepsilon _{\ast }\lambda_{\mbox{Leb}} (A\cap B_{r_{\ast }}(z_{\ast ,t}))  \label{hyp:lebesgue_bounded}
\end{align}%
where $\lambda_{\mbox{Leb}} $ is the Lebesgue measure on $\mathbb{R}^{N}.$
\end{enumerate}

We introduce a final structural assumption specifying that the time step $\delta$ needs to be small enough. For $\delta \in (0,1]$, when (\ref{eq:hyp_1_Norme_adhoc_fonction_schema}) holds, we define
\begin{align}
\label{def:eta_delta}
\eta_{1}(\delta)  :=  & \delta^{-d\frac{44}{91}} \min( 1,\frac{10^{d}}{m_{\ast}^{d} \vert 2^{10} (1+T^{3})  \vert ^{\frac{d}{2}}}) \quad \mbox{and} \\
 \eta_{2}(\delta)  :=  & \min(\delta^{-\frac{1}{2}} \eta_{1}(\delta)^{-\frac{1}{d}}, \frac{1}{2} \vert \delta^{\frac{1}{2}} 8  \mathfrak{D}_{3} \vert^{-\frac{1}{\mathfrak{p}_{3}+1}}) \nonumber.
\end{align}
with $\mathfrak{D}_{3},\mathfrak{p}_{3}$ given in (\ref{eq:hyp_1_Norme_adhoc_fonction_schema}).
We introduce the following assumption:
\begin{enumerate}[label=$\mathbf{A}_{5}$.]
\item \label{hyp:hyp_5_loc_var} Assume that (\ref{eq:hyp_1_Norme_adhoc_fonction_schema}) and $\mathbf{A}_{2}^{\infty}(L)$ (see (\ref{hyp:unif_hormander})) hold and that $\delta \in (0,1]$ is small enough so that 
\begin{align*}
\eta_{1}(\delta) > & \max(1, \frac{2^{1-\frac{d}{2}}}{d^{-\frac{d}{2}}}, 2(\frac{T \mathcal{V}_{L}^{\infty} m_{\ast}}{40(L+1) N^{\frac{L(L+1)}{2}} })^{-d13^{L}} ,\\
& 2 \mathbf{1}_{L=0}+2 \mathbf{1}_{L>0} \vert m_{\ast}\frac {\vert 2^{8} (1+T) \vert^{-143}}{10N^{\frac{L(L-1)}{2}}} \vert^{-d13^{L-1}} ).
\end{align*}
and $\eta_{2}(\delta)>1$ where those quantities are defined in (\ref{def:eta_delta}).
\end{enumerate}
\subsection{An alternative regularization property}
For $T \in \pi^{\delta}$, $\theta>0$, and $G$ a d-dimensional Gaussian random variable with mean 0 and covariance identity and independent from $(Z^{\delta}_{t})_{t \in \pi^{\delta,\ast}}$, we define
\begin{align*}
Q_{s,t}^{\delta,\theta }f(x)  =\int_{\mathbb{R}^{d}} f(y) Q^{\delta,\theta}_{s,t}(x,\mbox{d} y)  :=  \mathbb{E}[f(X^{\delta}_{t}+\delta^{\theta }G) \vert X^{\delta}_{s} = x].
\end{align*}
 The following result is a a direct application of Theorem 2.1 in \cite{Rey_2024}
\begin{proposition}
\label{th:regul_main_result_intro}
Let $L \in \mathbb{N}$ and let $f \in \mathcal{C}_{b}^{\infty}(\mathbb{R}^{d} ; \mathbb{R})$.

 Then we have the following properties:

\begin{enumerate}[label=\textbf{\Alph*.}]
\item \label{th:reg_gauss_regprop_intro} Let $q \in \mathbb{N}$, let $\gamma,\zeta \in \mathbb{N}^{d}$ such that $\vert \gamma
\vert +\vert \zeta \vert \leqslant q$. Assume that $\mathbf{A}^{\delta}_{1}( \max(q+3,2 L + 5))$ (see (\ref{eq:hyp_1_Norme_adhoc_fonction_schema})),  $\mathbf{A}_{2}^{\infty}(L)$ (see (\ref{hyp:unif_hormander})), $\mathbf{A}_{3}^{\delta}(+\infty)$ (see (\ref{eq:hyp:moment_borne_Z})), $\mathbf{A}_{4}^{\delta}$ (see (\ref{hyp:lebesgue_bounded})) and \ref{hyp:hyp_5_loc_var} hold.  Then,  for every $x \in \mathbb{R}^{d}$,
\begin{align}
  \label{eq:borne_semigroupe_regularisation_gauss_main} 
 \vert \partial_{x}^{\zeta}Q_{s,t}^{\delta,\theta }\partial_{x}^{\gamma}f(x) \vert \leqslant &  \Vert f \Vert_{\infty} \frac{C\exp(C (t-s)) }{\vert (t-s)   \mathcal{V}_{L}^{\infty} \vert^{\eta }}  ,
\end{align}
where $\eta \geqslant 0$ depends on $d,L,q$ and $\theta$ and $c,C \geqslant 0$ depend on $d,N,L,q,$$ \mathfrak{D}_{\max(q+3,2 L + 5)},$ $\mathfrak{p}_{\max(q+3,2 L + 5)},\frac{1}{m_{\ast}},\frac{1}{r_{\ast}},\theta$ and on the moment of $Z^{\delta}$ and which may tend to infinity if one of those quantities tends to infinity. \\
\item \label{th:reg_gauss_distprop_intro}  Assume that hypothesis from \ref{th:reg_gauss_regprop_intro} are satisfied with $\mathbf{A}^{\delta}_{1}( \max(q+3,2 L + 5))$ replaced by $\mathbf{A}^{\delta}_{1}(2 L + 5)$. Then, for every $x \in \mathbb{R}^{d}$,

\begin{align*}
\vert Q^{\delta}_{T}f(x)-Q_{T}^{\delta,\theta }f(x)\vert 
\leqslant &  \delta^{\theta}    \Vert f \Vert_{\infty} \frac{C\exp(C (t-s)) }{\vert  (t-s)  \mathcal{V}_{L}^{\infty} \vert^{\eta }} 
\end{align*}

where $\eta \geqslant 0$ depends on $d,L$ and $\theta$ and $C \geqslant 0$ depend on $d,N,L,q,$ $\mathfrak{D}_{2 L + 5},$ $\mathfrak{p}_{2 L + 5}$, $\frac{1}{m_{\ast}},\frac{1}{r_{\ast}},\theta$ and on the moment of $Z^{\delta}$ and which may tend to infinity if one of those quantities tends to infinity. 

\end{enumerate}
\end{proposition}

\subsection{Total variation convergence result}
Using the family of approximation schemes $(X^{\delta})_{\delta>0}$, for every $\nu \in \mathbb{N}^{\ast}$, we consider $\hat{Q}^{\nu,\delta_{n}^{0}}_{0,T}$ defined as in (\ref{eq:schema_ordre_quelconque}) and $P_{T}f(x) := \mathbb{E}[f(X_{T}) \vert X_{0}=x]$ where $X$ is the solution to (\ref{eq:SDE_inv_strato}). Now we are able to prove our main result.

\begin{theorem}
\label{th:main_result_psi_CVTV}
Let $T>0$, $n\in \mathbb{N}^{\ast}$, $\nu>0$ and $q_{\nu}=\max_{i\in \{1,\ldots,m(0,\nu) \}}(i\max(\beta,\kappa(1,q_i(\nu,0)))$. \\
\item We assume that $\mathbf{A}^{\delta}_{1}(+\infty)$ (see (\ref{eq:hyp_1_Norme_adhoc_fonction_schema})),  $\mathbf{A}_{2}^{\infty}(L)$ (see (\ref{hyp:unif_hormander})), $\mathbf{A}_{3}^{\delta}(+\infty)$ (see (\ref{eq:hyp:moment_borne_Z})), $\mathbf{A}_{4}^{\delta}$ (see (\ref{hyp:lebesgue_bounded})) and \ref{hyp:hyp_5_loc_var}  hold with $\delta$ replaced by $\delta^{1}_n$. Moreover, we assume that for every $k \in \mathbb{N}$, $k \geqslant n$, (\ref{hyp:transport_regularite_semigroup_approx}) and (\ref{hyp:transport_regularite_semigroup_approx_dual}) hold with $n$ replaced by $k$ and that the short time estimates $E_k(l,\alpha,\beta,P,Q)$ (see (\ref{hyp:erreur_tems_cours_fonction_test_reg})), and $E_k(l,\alpha,\beta,P,Q)^{\ast}$ (see (\ref{hyp:erreur_tems_cours_fonction_test_reg_dual})) hold for every $l \in \{1,\ldots,l(\nu,\alpha)\}$ if $k=n$ and for $l=l(\nu,\alpha)$ if $k > n$. If $n$ is large enough, then for every $f \in \mathcal{M}_{b}(\mathbb{R}^{d})$,
\begin{align}
\Vert P_{T}f-\hat{Q}^{\nu,\delta_{n}^{0}}_{0,T}f\Vert
_{\infty }\leqslant \frac{1}{n^{\nu}} \Vert f\Vert _{\infty }   \frac{C \exp(CT)}{(
\mathcal{V}_{L}^{\infty}T(\nu))^{\eta }} .
 \label{eq:distance_convergence_total_variation}
\end{align}
where $T(\nu)=\inf \left\{t \in \pi^{\delta_{n}^{1}},t \geqslant T\frac{n-m(0,v)}{n(m(0,v)+1)} \right\}$ and $\eta , C\geqslant 0$ do not depend on $n$ or $f$.

\end{theorem}

\begin{proof}
We have proved in Proposition \ref{th:regul_main_result_intro} that $Q^{\delta,\theta}$ verifies the regularization property. We chose $\theta \geqslant \nu+m(0,\nu)$ and the proof of (\ref{eq:distance_convergence_total_variation}) is then an immediate consequence  of Theorem \ref{theo:distance_density}. 
\end{proof}

\begin{example}

Let us consider $X=(X^{1},X^{2})$, the solution of the 2-dimensional system of $\mathbb{R}$ valued SDE, starting at point $x_{0}=(x^{1}_{0},x^{2}_{0}) \in \mathbb{R}^{2}$ and given by
\begin{align*}
d X^{1}_{t} =& b(X^{1}_{t},t) \mbox{d} t +  \sigma(X^{1}_{t},t)  \mbox{d} W_{t} \\
d X^{2}_{t} =& X^{1}_{t} \mbox{d} t  
\end{align*}
where $(W_{t})_{t \geqslant 0}$ is a one dimensional standard Brownian motion, $b$ and $\sigma$ continuous and bounded and their derivatives of any order are also continuous and bounded.  In the setting from (\ref{eq:SDE_inv}), we have $V_{0}:(x,t) \mapsto (b(x^{1},t),x^{1})$ and $V_{1}:(x,t) \mapsto (\sigma(x^{1},t),0)$. In this example uniform ellipticity holds for $X^{1}$ as long as $\inf_{(x^{1},t)\in \mathbb{R} \times \mathbb{R}_{+}} \sigma(x^{1},t)^{2} \neq 0$. However ellipticity does not hold for $X$ since $\mbox{dim} ( \mbox{span} ((\sigma,0)))(x,t) \leqslant 1<2$ for any $(x,t) \in  \mathbb{R}^{2} \times \mathbb{R}_{+}$. Nevertheless, let us compute the Lie brackets. In particular 
\begin{align*}
[V_{0},V_{1}]:(x,t) \mapsto (\partial_{x^{1}} \sigma (x^{1},t) b (x^{1},t)- \partial_{x^{1}} b(x^{1},t) \sigma(x^{1},t) , -\sigma(x^{1},t) 
),
\end{align*}
and,  for $\sigma(x^{1},t) \neq 0$, $\mbox{span} ((\sigma,0),(\partial_{x^{1}} \sigma  b - \partial_{x^{1}} b \sigma+\partial_{t} \sigma , - \sigma )(x,t) = \mathbb{R}^{2}$ so that local weak H\"ormander condition holds. Now, let us consider the Euler scheme of 
$X$, given by $(X^{\delta,1}_{0} ,X^{\delta,2}_{0} )=x_{0}$ and  for $t \in \pi^{\delta}$,
\begin{align*}
X^{\delta,1}_{t+\delta} =& X^{\delta,1}_{t} +b(X^{\delta,1}_{t},t) \delta +  \sigma(X^{1}_{t},t) \sqrt{\delta} Z^{\delta}_{t+\delta}\\
X^{\delta,2}_{t+\delta} =& X^{\delta,2}_{t}+ X^{\delta,1}_{t} \delta ,
\end{align*}
where $Z^{\delta}_{t} \in \mathbb{R}$, $t \in \pi^{\delta,\ast} $, are centered with variance one and Lebesgue lower bounded distribution and moment of order three equal to zero. With notations introduced in (\ref{def:hormander_func}), for $\sigma(x^{1},t) \neq 0$,
\begin{align*}
&\mathcal{V}_{1}(x,t) \\
&=  1 \wedge \inf_{\mathbf{b} \in \mathbb{R}^{d}, \vert \mathbf{b} \vert_{\mathbb{R}^{d}} =1}  \langle V_{1}(x,t) , \mathbf{b} \rangle_{\mathbb{R}^{d}}^{2} + \langle [V_{0}- \frac{1}{2} \nabla_{x}V_{1} V_{1},V_{1}](x,t) +\partial_{t}V_{1}(x,t) , \mathbf{b} \rangle_{\mathbb{R}^{d}}^{2}  \\
&= 1 \wedge \inf_{\mathbf{b} \in \mathbb{R}^{d}, \vert \mathbf{b} \vert_{\mathbb{R}^{d}} =1}  \langle (\sigma,0) , \mathbf{b} \rangle_{\mathbb{R}^{d}}^{2} + \langle (\partial_{x^{1}} \sigma b-\partial_{x^{1}} b \sigma+\frac{1}{2} \sigma^{2} \partial_{x^{1}}^{2}  \sigma+\partial_{t}\sigma,-\sigma) , \mathbf{b} \rangle_{\mathbb{R}^{d}}^{2}(x^{1},t) \\
&>0 .
\end{align*}

For a fixed time step $\delta>0$, we introduce 
\begin{align*}
\forall s,t \in \pi^{\delta}, s \leqslant t, \qquad  Q^{\delta}_{s,t}f(x) :=  \mathbb{E}[f(X^{\delta}_{t}) \vert X^{\delta}_{s}=x] .
  \end{align*}
Now, given $T>0$, $n\in \mathbb{N}^{\ast}$, $\nu>0$, we build $\hat{Q}^{\nu,\delta_{n}^{0}}_{0,T}$ using (\ref{eq:schema_ordre_quelconque}) and for $n$ large enough, for every $f \in \mathcal{M}_{b}(\mathbb{R}^{d})$, we have
\begin{align*}
\vert \mathbb{E}[f(X_{T}) \vert X_{0}=x ]-\hat{Q}^{\nu,\delta_{n}^{0}}_{0,T} \vert  \leqslant & \frac{1}{n^{\nu}} \Vert f \Vert_{\infty} \frac{  C\exp(C T) }{\vert  \mathcal{V}_{1}^{\infty} T(\nu) \vert^{\eta}}   .
\end{align*}
where $C$ and $\eta$ do not depend on $n$ or $f$.
\end{example}

\bibliography{Biblio}

\def\cprime{$'$} \def\cprime{$'$}
\begin{thebibliography}{10}

\bibitem{AlGerbi_Jourdain_Clement_2016}
A.~Al~Gerbi, B.~Jourdain, and E.~Clément.
\newblock Ninomiya–victoir scheme: Strong convergence, antithetic version and
  application to multilevel estimators.
\newblock {\em Monte Carlo Methods and Applications}, 22(3):197--228, 2016.

\bibitem{Alfonsi_2010}
A.~Alfonsi.
\newblock High order discretization schemes for the {CIR} process: application
  to affine term structure and {H}eston models.
\newblock {\em Math. Comp.}, 79(269):209--237, 2010.

\bibitem{Alfonsi_Bally_2019}
A.~Alfonsi and V.~Bally.
\newblock A generic construction for high order approximation schemes of
  semigroups using random grids.
\newblock {\em Numerische Mathematik}, 148:743 -- 793, 2019.

\bibitem{Bally_Rey_2015}
V.~Bally and C.~Rey.
\newblock {Approximation of Markov semigroup in total variation disctance}.
\newblock January 2015.

\bibitem{Bally_Talay_1996_I}
V.~Bally and D.~Talay.
\newblock The law of the {E}uler scheme for stochastic differential equations.
  {I}. {C}onvergence rate of the distribution function.
\newblock {\em Probab. Theory Related Fields}, 104(1):43--60, 1996.

\bibitem{BenAlaya_Kebaier_2015}
M.~Ben~Alaya and A.~Kebaier.
\newblock {Central limit theorem for the multilevel Monte Carlo Euler method}.
\newblock {\em The Annals of Applied Probability}, 25(1):211 -- 234, 2015.

\bibitem{Giles_2008}
M.B. Giles.
\newblock Multilevel monte carlo path simulation.
\newblock {\em Oper. Res.}, 56:607--617, 2008.

\bibitem{Glynn_Rhee_2015}
P.W. Glynn and C-H Rhee.
\newblock Unbiased estimation with square root convergence for sde models.
\newblock {\em Operations Research}, 63(5):1026--1043, 2015.

\bibitem{Heinrich_2001}
S.~Heinrich.
\newblock Multilevel monte carlo methods.
\newblock In {\em Large-Scale Scientific Computing}, pages 58--67, Berlin,
  Heidelberg, 2001. Springer Berlin Heidelberg.

\bibitem{Kebaier_2005}
A.~Kebaier.
\newblock Statistical romberg extrapolation: A new variance reduction method
  and applications to option pricing.
\newblock {\em The Annals of Applied Probability}, 15(4):2681--2705, 2005.

\bibitem{Lemaire_Pages_2017}
V.~Lemaire and G.~Pag{\`e}s.
\newblock {Multilevel Richardson–Romberg extrapolation}.
\newblock {\em Bernoulli}, 23(4A):2643 -- 2692, 2017.

\bibitem{McLeish_2011}
D.~McLeish.
\newblock A general method for debiasing a monte carlo estimator.
\newblock {\em Monte Carlo Methods and Applications}, 17(4):301--315, 2011.

\bibitem{Ninomiya_Victoir_2008}
S.~Ninomiya and N.~Victoir.
\newblock Weak approximation of stochastic differential equations and
  application to derivative pricing.
\newblock {\em Appl. Math. Finance}, 15(1-2):107--121, 2008.

\bibitem{Pages_2007}
G.~Pagès.
\newblock Multi-step richardson-romberg extrapolation: Remarks on variance
  control and complexity.
\newblock {\em Monte Carlo Methods and Applications}, 13(1):37--70, 2007.

\bibitem{Rey_2017}
C.~Rey.
\newblock Convergence in total variation distance of a third order scheme for
  one-dimensional diffusion processes.
\newblock {\em Monte Carlo Methods and Applications}, 23(1):1--12, 2017.

\bibitem{Rey_2024}
C.~Rey.
\newblock H\"ormander properties of discrete time markov processes, 2024.

\bibitem{Talay_1990}
D.~Talay.
\newblock Second-order discretization schemes of stochastic differential
  systems for the computation of the invariant law.
\newblock {\em Stochastics and Stochastic Reports}, 29(1):13--36, 1990.

\bibitem{Talay_Tubaro_1991}
D.~Talay and L.~Tubaro.
\newblock Expansion of the global error for numerical schemes solving
  stochastic differential equations.
\newblock {\em Stochastic Anal. Appl.}, 8(4):483--509 (1991), 1990.

\bibitem{Vihola_2018}
M.~Vihola.
\newblock Unbiased estimators and multilevel monte carlo.
\newblock {\em Operations Research}, 66(2):pp. 448--462, 2018.

\end{thebibliography}
\bibliographystyle{plain}

\end{document}